%% file: FPsystem_revised_version.tex
\documentclass[a4paper]{amsart}
\usepackage{a4}
\usepackage{amsmath}
\usepackage{amssymb}
\usepackage[applemac]{inputenc}
\usepackage{graphicx}
\usepackage{color}
\newcommand{\PP}{ \mathbb{P}}

\newcommand{\M}{{\mathbb M}}

\newcommand{\EE}{{\mathbb E}}

\newcommand{\T}{{\mathbb{T}}}

\newtheorem{remark}{\textbf{Remark}}[section]

\newtheorem{theorem}{\textbf{Theorem}}[section]

\newtheorem{proposition}{\textbf{Proposition}}[section]

\numberwithin{equation}{section}

\title[A fully-discrete scheme for systems of nonlinear FPK equations]{A fully-discrete scheme for systems of nonlinear Fokker-Planck-Kolmogorov equations} 

\author{
Elisabetta Carlini  \and     Francisco J. Silva   }
\input mymacro

\thanks{  ``Sapienza'', Universit\`a di Roma, Dipartimento di Matematica Guido Castelnuovo, 00185 Rome, Italy (carlini@mat.uniroma1.it).}

\thanks{ Institut de recherche XLIM-DMI, UMR-CNRS 7252 Facult\'e des sciences et techniques 
Universit\'e de Limoges, 87060 Limoges, France and TSE-R, Universit\'e Toulouse I Capitole  (francisco.silva@unilim.fr) }

\begin{document}
\maketitle
\begin{abstract} We consider a system of Fokker-Planck-Kolmogorov (FPK) equations, where the dependence of the coefficients is nonlinear and nonlocal in time with respect to the unknowns. We extend the numerical scheme proposed and studied in \cite{Carlini_Silva_2017} for a single FPK equation of this type. We analyse the convergence of the scheme and we study its applicability in two examples. The first one concerns a population model involving two interacting species and the second one concerns two populations Mean Field Games. 
\end{abstract}\vspace{0.5cm}

\small {\bf AMS-Subject Classification:}  35Q84,	 65N12, 65N75.  \vspace{0.2cm}

\normalsize

\section{Introduction} In this note we consider the following system of nonlinear Fokker-Planck-Kolmogorov (FPK) equations  
$$\displaystyle
\ba{rcl} \partial_{t} m^{\ell}  -\half \underset{1\leq i,j \leq d}\sum \partial_{x_i,x_j}^2\left(a_{i,j}^{\ell}(m, x,t)m^{\ell}\right)+\sum_{i=1}^{d_\ell}\partial_{x_i}\left(b^{\ell}(m,x,t)m^{\ell}\right) &=& 0,  \hspace{0.5cm} \mbox{in } \; \RR^{d_\ell} \times (0,T), \\[4pt]
													m^{\ell}(0)&=&\bar{m}_{0}^{\ell}  \hspace{0.5cm} \mbox{in } \; \RR^{d_\ell},
\ea \hspace{0.3cm}\eqno(FPK)$$  
where $\ell=1,\hdots, M$ and $d_\ell \in \NN \setminus \{0\}$. 
In the system above, we look for $M$ unknowns $m=(m^1,\hdots, m^{M})$ such that for each $\ell=1,\hdots, M$,  $m^\ell$ belongs to the space $C([0,T]; \P_1(\RR^{d_\ell}))$, where $\P_1(\RR^{d_\ell})$ is the set of probability measures on $\RR^{d_\ell}$ with finite first order moment. This set is endowed with the standard Monge-Kantorovic distance (see Section \ref{preliminaries}  below). The coefficients in $(FPK)$ are given by functions 
$$\displaystyle
b^{\ell}: \prod_{\ell'=1}^{M}C([0,T];\P_1(\RR^{d_{\ell'}})) \times \RR^{d_\ell} \times [0,T] \to \RR^{d_\ell}, \; \; \; \; \; a_{i,j}^{\ell}= \sum_{p=1}^{r_\ell}\sigma^{\ell}_{i,p} \cdot \sigma^{\ell}_{j,p} \hspace{0.3cm} \forall \; i, \; j=1,\dots d,
$$
where  $r_\ell \in \NN \setminus \{0\}$ and for all $p=1, \hdots, r_\ell$ 
$$\displaystyle
 \sigma^{\ell}_{i,p}: \prod_{\ell'=1}^{M}C([0,T];\P_1(\RR^{d_{\ell'}}))  \times \RR^{d_\ell} \times [0,T] \to \RR.
$$
Finally, the prescribed initial distributions $\bar{m}_0:= (\bar{m}_0^1,\hdots,\bar{m}_0^N)$ are assumed to be probability measures with finite second order moments, i.e. $\int_{\RR^{d_\ell}} |x|^2 \dd \bar{m}_0^\ell <\infty$ for all $\ell=1, \hdots, M$. Note that system $(FPK)$ is highly nonlinear because the dependence on $m$ of the coefficients $b^{\ell}$ and $a_{i,j}^{\ell}$ can be nonlocal  in time. A priori these coefficients depend of the entire trajectory $t \in [0,T] \to m(t) \in  \prod_{\ell=1}^{M} \P_1(\RR^{d_{\ell}})$.

When $M=1$, and the coefficients $b^1$ and $\sigma^1$ do not depend on $m$, the resulting equation is the classical FPK equation that describes the law of a diffusion process whose drift and volatility coefficients are given by $b^1$ and $\sigma^1$, respectively. We refer the reader to the monograph \cite{MR3443169} for a rather complete account of analytical results related to this equation and to the references in introduction of \cite{Carlini_Silva_2017} for the numerical approximation of its solutions. 

Let us now comment on the probabilistic interpretation of $(FPK)$ when $M>1$. Formally speaking, provided that for each $\ell=1,\hdots,M$, the equation 
\be\label{diffusions_indexed_by_ell}\displaystyle
\dd X_{\ell}(t) = b^{\ell}( m,X_{\ell}(t), t) \dd t  + \sum_{p=1}^{r_\ell}\sigma_{\cdot, p}^{\ell}  (m,X_{\ell}(t), t) \dd W^{\ell}_{p}(t) \; \; t\in [0,T], \; \; \; X_{\ell}(0)=X_0^\ell,
\ee
is well-posed (let us say in a weak sense),   system $(FPK)$ describes the time evolution of  the laws of  $ [0, T] \ni t  \mapsto X_{\ell}(t)\in \RR^{d_\ell}$. In \eqref{diffusions_indexed_by_ell}, the Brownian motions $\{W^{\ell}_p \; ; \; p=1, \hdots, M, \; \; m=0,\hdots, r_\ell\}$ are mutually independent and independent of $(X_0^\ell)_{\ell=1}^{M}$, where, for each $\ell$, the distribution of $X_0^\ell$ is given by $m_{0}^{\ell}$. In addition, the map $m: [0,T] \to \Pi_{\ell =1}^{M} \P_{1}(\RR^{d_{\ell }})$ is given  by $m(t)= \left(\mbox{Law}(X_1(t)), \hdots, \mbox{Law}(X_M(t))\right)$. 

Our aim in this paper is to use this probabilistic interpretation  in order to provide a convergent fully discrete scheme for $(FPK)$. The analysis of the proposed approximation, that we will present in Section \ref{fully_discrete_scheme}, is a rather straightforward extension of the study done in \cite{Carlini_Silva_2017}, where $M=1$.  On the other hand,
as we will show in the next section, it is easy to see that solutions of $(FPK)$ can be found as the marginal laws of a single FPK equation whose solution takes values in $\P_{1}(\prod_{\ell=1}^{M}\RR^{d_\ell})$ at each time. Therefore,   the scheme in \cite{Carlini_Silva_2017} could, in principle, be used  to approximate $(FPK)$.   However, from the practical point of view, this roadmap has serious difficulties because the numerical efficiency of the scheme in   \cite{Carlini_Silva_2017} depends heavily on  the dimension of the state space. In this sense, the study of a scheme that can be directly applied to system $(FPK)$ is interesting in its own right. 

We implement the scheme in two examples. In the first one we consider a diffusive version, introduced in \cite{CarLab15}, of a system of FPK equations proposed in \cite{DiFrancescoFagioli13} modelling the evolution of two interacting species under attraction and repulsion effects.  Since in \cite{CarLab15} some of the drift terms depend on the densities of the species distributions, we need to regularize these terms in order to obtain a convergent approximation in our framework. Our discretization produces  rather similar numerical results to those in \cite[Section 5.1]{BCL16}. In the second example, we consider a particular instance of a two population Mean Field Game (MFG) (see e.g. \cite{MR3333058}). The system we consider, introduced in \cite[Section 6.2.1]{AcBarCi17}, is symmetric with respect to both populations and aims to model xenophobia effects on urban settlements. In \cite{AcBarCi17} it is shown that even if at the microscopic level the xenophobic effect is small, segregation occurs at the macroscopic level, indicating  that  Schelling's principle (see \cite{Schelling1}) is also valid in the context of MFGs.  In the tests that we have implemented, we recover the numerical results in \cite{AcBarCi17} for the viscosity parameters the authors consider, but we are also able to deal with very small, or null, viscosity parameters, capturing, for these cases, different segregated configurations than those in  \cite{AcBarCi17}.  We believe that the possibility of dealing with small or null viscosity parameters, as well as large time steps, is an important feature of the scheme that we propose.

The article is organized as follows. In the next section we introduce some standard notations and  our main assumptions. In Section \ref{fully_discrete_scheme}  we introduce the scheme that we propose, which is a straightforward extension of the one in \cite{Carlini_Silva_2017}, and we study its main properties, including the convergence analysis. Finally, in Section \ref{simultions}, we present our numerical results for the two examples described in the previous paragraph. \bigskip

{\bf Acknowledgements:}  The first author acknowledges financial support by the Indam GNCS project ``Metodi numerici per equazioni iperboliche e cinetiche e applicazioni''.  The second author is partially supported by the ANR project
MFG ANR-16-CE40-0015-01 and the PEPS-INSMI Jeunes project ``Some open problems in Mean Field Games'' for the years 2016 and 2017.

Both authors acknowledge financial support by the PGMO project VarPDEMFG. 
\section{Preliminaries and main assumptions}\label{preliminaries}
Let us first set some standard notations and assumptions that we will use in the rest of the paper. For the sake of notational convenience we will assume   that $M=2$, but our results admit straightforward generalizations for arbitrary $M \in \NN$.  The set $\P_{i}(\RR^d)$ ($d$, $i \in \NN \setminus \{0\})$ denotes the set of Borel probability measures over $\RR^d$ with finite  $i$-th order moment. We endow $\P_{i}(\RR^d)$ with the standard Monge-Kantorovic metric
$$\displaystyle
d_{i}(\mu_1, \mu_2):= \inf\left\{ \left( \int_{\RR^d \times \RR^d} |x-y|^i \dd \gamma(x,y) \right)^{\frac{1}{i}} \; \big| \; \Pi_{x} \sharp \gamma = \mu_1, \; \; \Pi_{y} \sharp \gamma = \mu_2 \right\},
$$
where $\Pi_{x}(x,y):=x$, $\Pi_{y}(x,y):=y$ for all $x$, $y\in \RR^d$ and given a Borel map $\Phi: \RR^{m} \to \RR^{n}$ and a Borel measure $\mu$ on $\B(\RR^{m})$, the {\it push-forward} measure $\Phi \sharp \mu$ is defined as $\Phi \sharp \mu(A) := \mu( \Phi^{-1}(A))$. Let $\K \subseteq \P_{i}(\RR^d)$ be given. A useful compactness result in $\P_{i}(\RR^d)$  states that if for a given $\K\subseteq \P_{i}(\RR^d)$  there exists $C>0$ such that   
\be\label{condition_for_compactness}\int_{\RR^d} |x|^{i+\delta} \dd \mu(x)\leq C \qquad \mbox{for some $\delta>0$ and all $\mu \in \K$},
\ee
then $\K$ is  relatively compact (see e.g. \cite[Proposition 7.1.5]{Ambrosiogiglisav}).

Define $\M:= C([0,T];\P_1(\RR^{d_{1}}))\times  C([0,T];\P_1(\RR^{d_{2}}))$. We say that $m=(m^1,m^2) \in  \M $ is a weak solution of $(FPK)$ if for all $\ell=1,2$, $t\in [0,T]$ and $\phi \in C_{0}^{\infty}(\RR^{d_\ell})$ (the space of  $C^{\infty}$ real-valued functions defined on $\RR^{d_\ell}$ and with compact support) we have that
\be\label{solution_FP}\displaystyle
\ba{ll}
\int_{\RR^{d_\ell}} \varphi(x) \dd m^\ell(t)(x)=& \int_{\RR^{d_\ell}}  \varphi(x) \dd \bar{m}_{0}^\ell(x) + \int_{0}^{t}\int_{\RR^{d_\ell}} \left[   b^\ell(m,x,s) \cdot \nabla \varphi(x)   \right] \dd m^\ell(s)(x) \dd s \\[6pt]
\;  & + \int_{0}^{t}\int_{\RR^{d_\ell}} \left[  \half \sum_{i,j} a_{i,j}^\ell(m,x,s) \partial_{x_i, x_j}^{2} \varphi(x)\right] \dd m^\ell(s)(x) \dd s,
\ea\ee 
provided that the second and third terms in the right hand side are meaningful. 

 The main assumptions in this paper are continuity and uniform linear growths of $b^\ell$ and $\sigma^\ell$, respectively, with respect to the space variables. More precisely, we will suppose that \smallskip\\
{\bf(H)} For $\ell=1,2$ \smallskip\\
{\rm(i)} $\bar{m}_{0}^\ell \in \P_2(\RR^d)$.\smallskip\\
{\rm(ii)} The maps $b^\ell$ and $\sigma^\ell$ are continuous. \smallskip\\
{\rm(iii)} There exists $C>0$ such that  
\be\label{linear_growth}\displaystyle
|b^\ell(m,x,t)| +|\sigma^\ell(m,x,t)| \leq C(1+ |x|) \hspace{0.2cm} \forall \;  m \in \M, \; \; x\in \RR^{d_\ell}, \;  t\in [0,T].
\ee \smallskip


Note that system $(FPK)$ can be analysed with the help of a single FPK equation. Indeed,  let  $\bar{m}_{0}\in \P_2( \RR^{d_1} \times \RR^{d_2})$ be such that its   marginal  in $\RR^{d_\ell}$ ($\ell=1,2$) is given by  $\bar{m}_0^{\ell}$. Given  $\mu \in C([0,T]; \P_1(\RR^{d_1} \times \RR^{d_2}))$    denote by $\hat{\mu} :=(\mu^{1} , \mu^2) \in \M$ the marginals in $\RR^{d_1}$ and  $\RR^{d_2}$  of $t \in [0,T] \to \mu(t)\in \P_1( \RR^{d_1} \times \RR^{d_2})$. Writing    $x=(x^{1}, x^2) \in \RR^{d_1} \times \RR^{d_2}$ define the coefficients
$$ \displaystyle
\ba{l}
b: C\left([0,T];\P_{1}(\RR^{d_1} \times \RR^{d_2})\right) \times (\RR^{d_1} \times \RR^{d_2}) \times [0,T] \to   \RR^{d_1} \times \RR^{d_2}, \\[6pt]
\sigma: C\left([0,T];\P_{1}(\RR^{d_1} \times \RR^{d_2})\right) \times  (\RR^{d_1} \times \RR^{d_2}) \times [0,T] \to  \RR^{d_1 \times r_1} \times \RR^{d_2 \times r_2}.
\ea
$$
as
\be\label{redefinicion_coeficientes} \displaystyle
b(\mu,x,t)  :=  \left(b^{1}(\hat{\mu},x^{1},t), b^{2}(\hat{\mu},x^{2},t)\right), \hspace{0.3cm}
\sigma(\mu,x,t) :=  \left(\sigma^{1}(\hat{\mu},x^{1},t),   \sigma^{2}(\hat{\mu},x^{2},t)\right), 
\ee
for all $\mu \in  C\left([0,T]; \P_1(\RR^{d_1} \times \RR^{d_2})\right)$,  $x \in \RR^{d_1} \times \RR^{d_2}$ and $t\in [0,T]$. Finally, for all $\ell_1$, $\ell_2= 1, 2$ we  set
$$\displaystyle
a_{i,j}^{\ell_1,\ell_2}(\mu, x,t):=\left\{ \ba{lr}\sum_{p=1}^{d_{\ell_1}}\sigma^{\ell_1}_{i,p}(\hat{\mu},x^{\ell_1},t)\sigma^{\ell_1}_{j,p} (\hat{\mu},x^{\ell_1},t)& \mbox{if $\ell_1=\ell_2$}, \\[5pt]
0 &  \mbox{if $\ell_1 \neq \ell_2$.}
\ea \right.
$$
Consider the problem of finding $m \in C\left([0,T]; \P_1(\RR^{d_1} \times \RR^{d_2})\right)$  such that  
$$\displaystyle
\ba{rcl} \partial_{t} m   -\half \underset{\begin{subarray}{c}
   1\leq \ell_1, \ell_2 \leq 2 \\
  1\leq i, j \leq d_\ell
  \end{subarray}}\sum \partial_{x_i^{\ell_1},x_j^{\ell_2}}^2\left(a_{i,j}^{\ell_1,\ell_2}(m, x,t)m\right)+\mbox{div}\left(b(m,x,t)m\right) &=& 0  \hspace{0.5cm} \mbox{in } \; \RR^{d_1} \times \RR^{d_2}  \times [0,T], \\[4pt]
													m(0)&=&\bar{m}_{0}  \hspace{0.5cm} \mbox{in } \; \RR^{d_1} \times \RR^{d_2}.
\ea\eqno(FPK')$$ \normalsize
If {\bf(H)} holds, then the coefficients $b$ and $\sigma$, defined in \eqref{redefinicion_coeficientes},  also satisfy {\bf(H)} in the corresponding spaces. More precisely,  $b$ and $\sigma$ are continuous and there exists $C>0$ such that 
\be\label{linear_growth_1}|b(m,x,t)| +|\sigma(m,x,t)| \leq C(1+ |x|) \hspace{0.2cm} \forall \;  m \in C\left([0,T];\P_{1}(\RR^{d_1} \times \RR^{d_2})\right), \;   x\in \RR^{d_1} \times \RR^{d_2}, \;  t\in [0,T].\ee
Thus, by the results in \cite{MR3086740,MR3113428} (see also \cite[Theorem 4.2]{Carlini_Silva_2017}) we have that $(FPK')$ admits at least one solution $m \in C\left([0,T];\P_1(\RR^{d_1} \times \RR^{d_2})\right)$. Moreover, from the results in \cite{Carlini_Silva_2017} we have the existence of $C>0$ such that 
\be\label{uniform_quadratic_growth_1}
\sup_{t\in [0,T]} \int_{\RR^{d_1} \times \RR^{d_2}} |x|^{2} \dd m(t)(x) \leq C. 
\ee

Now,  for $R>0$ and $x'\in \RR^{d_2}$  we set $\xi_{R}(x'):= \xi(x'/R)$, where $\xi \in C_{0}^{\infty}(\RR^{d_2})$ is such that $0\leq \xi \leq 1$, $\xi(x')= 1$ if $|x'| \leq 1/2$ and $\xi(x')=0$ if $|x'| \geq  1$. The function $\xi_R$ belongs to $C_{0}^{\infty}(\RR^{d_2})$ and,  as $R \uparrow \infty$,  approximate the constant function equal to $1$ in $\RR^{d_2}$.  Given  $\varphi\in C_{0}^{\infty}(\RR^{d_1})$, let us define $\varphi_R^1: \RR^{d_1} \times \RR^{d_2}  \to \RR$ as $\varphi_R^1(x):= \varphi(x^{1}) \xi_{R}(x^2)$, which belongs to $C_{0}^{\infty}(\RR^{d_1} \times \RR^{d_2})$. By considering this test function in $(FPK')$, using \eqref{linear_growth_1} and \eqref{uniform_quadratic_growth_1} and letting $R\uparrow \infty$ we obtain that $m^1 \in C([0,T];  \P_1(\RR^{d_1}))$ (defined for all $t\in [0,T]$ as the marginal of $m(t)$ with respect to $\RR^{d_1}$) satisfies \eqref{solution_FP} with $\ell=1$. A similar construction shows that $m^2 \in C([0,T]; \P_1(\RR^{d_2}))$ (defined for all $t\in [0,T]$ as the marginal of $m(t)$ with respect to $\RR^{d_2}$) satisfies \eqref{solution_FP} with $\ell=2$. As a result $(m^1,m^2)$ solves $(FPK)$. 

From the analytical point of view, the argument above is useful in order to obtain existence and properties of solutions to $(FPK)$. On the other hand, as we comment in Remark \ref{numerical_inefficiency_one_fp} in the next section, this simplification is useless from the numerical point of view. 
\section{The fully discrete scheme}\label{fully_discrete_scheme}  
We consider a time step $h=T/N^T$ ( $N^T \in \NN$) and  space steps $\rho^1$, $\rho^2>0$. We define $t_k=kh$ ($k=0,\hdots, N_T$), the time grid $\{0, t_1, \hdots, t_{N^T-1}, T\}$ and the space grids  $\G_{\rho^{\ell}}:= \{ x_i^\ell= \rho^{\ell} i \; | \; i\in \ZZ^{d_\ell}\}$ ($\ell=1$, $2$). 
We consider two regular lattices $\T_{\rho^1}$ and $\T_{\rho^2}$ of $\RR^{d_{1}}$ and $\RR^{d_2}$, where the vertices of the square elements belong to $\G_{\rho^{1}}$ and $\G_{\rho^{2}}$, respectively. Associated to these lattices  and their vertices, we consider two $\mathbb{Q}_1$ bases $\{ \beta_{i}^{\ell} \; ; i\in \ZZ^{d_{\ell}} \}$ ($\ell=1$, $2$) . By definition, for $\ell=1$, $2$ and $i \in \ZZ^{d_\ell}$, the functions  $\beta_{i}^\ell:\RR^{d_\ell} \to \RR_+$ (where $\RR_+$ denotes the set of non negative real numbers) are  polynomials of degree less than or equal to 1 with respect to each variable $(x_ 1, \dots, x_{d_{\ell}})$  on each square $Q \in \T_{\rho^\ell}$,  have compact support and satisfy that $\beta_{i}^\ell(x_j^\ell)=\delta_{i,j}$ (where $\delta_{i,j}=1$ if $i=j$ and   $\delta_{i,j}=0$, otherwise)  and $\sum_{i\in \ZZ^{d_\ell}} \beta^{\ell}_{i}(x)=1$ for all $x\in \RR^{d_\ell}$. In order to define a discretization of the initial condition $\bar{m}_0^\ell$ we define the sets 
$$\displaystyle
E_{i}^\ell:= \left\{ x \in \RR^{d_\ell} \, ; \; |x-x_{i}|_{\infty} \leq \frac{\rho^\ell}{2}\right\}.
$$
Since we will let $\rho^\ell$ tend to $0$ later, without loss of generality we can assume that $\bar{m}_0^\ell(\partial E_i^\ell)=0$ for all $i\in \ZZ^{d_\ell}$. We then set 
$$\displaystyle
m_{i,0}^\ell=  \bar{m}_{0}^\ell (E_i^\ell) \hspace{0.5cm} \forall \; i \in \ZZ^{d_\ell}.
$$ 
Since $\bar{m}_0^{\ell}(\RR^{d_\ell})=1$, we have that $\left\{m_{i,0}^\ell \; | \; i \in \ZZ^{d_\ell}\right\}$ belongs to the {\it simplex} 
$$
\mathcal{S}^{\rho_\ell}:= \left\{ \mu \in [0,1]^{\ZZ^{d_\ell}} \; | \; \sum_{i\in \ZZ^{d_\ell}} \mu_i=1 \right\}. 
$$
Given $\mu= \left\{ \mu_{i,k} \; | \; i \in \ZZ^{d_\ell}, \; k=0,\hdots, N_T \right\}\in (\mathcal{S}^{\rho_\ell})^{N_T+1}$, we identify $\mu$ with an element in $C([0,T]; \P_1(\RR^{d_\ell}))$ via a linear interpolation
\be\label{linear_interpolation_extension_of_simplex}
\mu(t):= \left(\frac{t-t_{k}}{h}\right)\sum_{i\in \ZZ^{d_\ell}} \mu_{i,k+1} \delta_{x_{i}^\ell} +  \left(\frac{t_{k+1}-t}{h}\right)\sum_{i \in \ZZ^{d_\ell}} \mu_{i,k}\delta_{x_{i}^\ell} \hspace{0.4cm} \mbox{if $t\in [t_{k}, t_{k+1}[$.}
\ee
Now, we have all the elements to introduce the discretization of $(FPK)$ we consider. For the sake of clarity, we first recall the fully-discrete scheme introduced in \cite{Carlini_Silva_2017} when $M=1$. In this case the $(FPK)$ system is given by 
\be\label{fpk_M_one}\ba{rcl} \partial_{t} m   -\half \underset{1\leq i,j \leq d}\sum \partial_{x_i,x_j}^2\left(a_{i,j} (m, x,t)m \right)+\sum_{i=1}^{d}\partial_{x_i}\left(b (m,x,t)m \right) &=& 0,  \hspace{0.5cm} \mbox{in } \; \RR^{d} \times (0,T), \\[4pt]
													m(0)&=&\bar{m}_{0}  \hspace{0.5cm} \mbox{in } \; \RR^{d},
\ea
\ee
where we have omitted the superfluous index $\ell=1$. The fully discrete scheme for \eqref{fpk_M_one} reads: Find $m \in (\mathcal{S}^{\rho})^{N_T+1}$ such that
\be\label{scheme_nonlinear_stochastic_case_II}\ba{rcl}
m_{i,0} &=& \bar{m}_{0} (E_{i})  \hspace{0.4cm} \forall \; i\in \ZZ^d,  \\[6pt]
m_{i, k+1} &=& \frac{1}{2r}\sum\limits_{p=1}^{r}\sum\limits_{j \in \ZZ^{d}} \left[\beta_{i} ( \Phi_{j,k}^{p,+}[m])+\beta_{i}(\Phi_{j,k}^{p,-}[m])\right] m_{j,k} \hspace{0.4cm} \forall \; i\in \ZZ^d, \; \; k=0,\hdots, N_T-1, 
\ea \ee 
where the {\it one-step} discrete characteristics starting from $x_{j}$ at time $t_k$ are defined as
$$
\Phi_{j,k}^{p,+}[m]  :=  x_{j} + h b(m,x_{j},t_{k})+ \sqrt{r h }\sigma_{p}(m,x_{j},t_{k}), \; \; \; 
\Phi_{j,k}^{p,-}[m]  : =  x_{j} + h b(m,x_{j},t_{k})-\sqrt{r h }\sigma_{p}(m, x_{j},t_{k}),
$$
with $b$ and $\sigma_p$ being defined, as a function of $m$, through the extension \eqref{linear_interpolation_extension_of_simplex}. 

Existence of at least one solution $m^{\rho,h}$ to \eqref{scheme_nonlinear_stochastic_case_II} has been proved in  \cite[Proposition 3.1]{Carlini_Silva_2017}. Moreover, under an additional local Lipschitz assumption on $b$ and $\sigma$,  as $\rho$ and $h$ tend to $0$ and $\rho^2=o(h)$, the sequence $m^{\rho,h}$ in $C([0,T]; \P_1(\RR^d))$, defined again  through the extensions \eqref{linear_interpolation_extension_of_simplex}, has at least one limit point $m\in C([0,T]; \P_1(\RR^d))$, and every such limit point solves $(FPK)$ (see \cite[Theorem 4.1]{Carlini_Silva_2017}).
\begin{remark}\label{extension_of_the_scheme_non_regular_case}
 By regularizing the coefficients $b$ and $\sigma$ using standard mollifiers, and modifying the scheme accordingly,  this convergence result is also shown to hold under assumption {\bf(H)} only {\rm(}see \cite[Theorem 4.2]{Carlini_Silva_2017}{\rm)}. 
\end{remark}

In order to grasp the probabilistic interpretation of  \eqref{scheme_nonlinear_stochastic_case_II}, it is useful to think this problem as the one of finding a fixed point of a suitable mapping. Indeed, given $\mu \in C([0,T]; \P_1(\RR^d))$ and a solution $m[\mu]\in (\mathcal{S}^{\rho})^{N_T+1}$  to
\be\label{scheme_nonlinear_stochastic_case_II_dependent_of_mu}\ba{rcl}
m_{i,0} &=& \bar{m}_{0} (E_{i})  \hspace{0.4cm} \forall \; i\in \ZZ^d,  \\[6pt]
m_{i, k+1} &=& \frac{1}{2r}\sum\limits_{p=1}^{r}\sum\limits_{j \in \ZZ^{d}} \left[\beta_{i} ( \Phi_{j,k}^{p,+}[\mu])+\beta_{i}(\Phi_{j,k}^{p,-}[\mu])\right] m_{j,k} \hspace{0.4cm} \forall \; i\in \ZZ^d, \; \; k=0,\hdots, N_T-1, 
\ea \ee 
we can construct a probability space $(\Omega, \F, \PP)$ and Markov chain $\{X_k[\mu] \; | \; k=0,\hdots, N_T\}$, defined on it,  taking values in $\G_{\rho}$  and whose marginal laws  and  transition probabilities are  given, respectively, by $m[\mu]_{(\cdot), k}\in \mathcal{S}^{\rho}$ and  
\be\label{transition_probabilities_M_one}
\PP\left( X_{k+1}[\mu]=x_{i} \; \big| \; X_{k}[\mu]=x_{j}\right)= \frac{1}{2r}\sum\limits_{p=1}^{r}\left[\beta_{i} ( \Phi_{j,k}^{p,+}[\mu])+\beta_{i}(\Phi_{j,k}^{p,-}[\mu])\right] \hspace{0.3cm} \; \forall \; i, j \in \ZZ^{d}, \; \; k=0,\hdots,N_T-1.
\ee \normalsize
In \cite{Carlini_Silva_2017} the Markov chain defined above is shown to satisfy the {\it consistency conditions} introduced by Kushner (see e.g. \cite{KushnerDupuis}). Hence, we can expect that its marginal laws will approximate the law of a weak solution $X[\mu]$ to 
\be\label{diffusions_dependending_on_mu_M_one}
\dd X(t) = b(\mu,X(t), t) \dd t  + \sum_{p=1}^{r}\sigma_{\cdot, p}   (\mu,X(t), t) \dd W_{p}(t) \; \; t\in [0,T], \; \; \; X(0)=X_0,
\ee
where the distribution of $X_0$ is given by $\bar{m}_0$. As explained in \cite{Carlini_Silva_2017}, a solution to $(FPK)$, when $M=1$, corresponds  to a fixed point $m\in C([0,T];\P_1(\RR^d))$ of the application $C([0,T];\P_1(\RR^d)) \ni \mu \to m[\mu](\cdot) \in C([0,T];\P_1(\RR^d))$, where, for every $t \in [0,T]$, the measure $m[\mu](t)$ is defined as the law of $X[\mu](t)$. Based on this interpretation, scheme \eqref{scheme_nonlinear_stochastic_case_II} can be interpreted as the analogous fixed point problem for the approximating Markov chain $\{X_k[\mu] \; | \; k=0,\hdots, N_T\}$.

Having the previous observations in mind, the extension of scheme \eqref{scheme_nonlinear_stochastic_case_II} to the case $M=2$ is straightforward. We consider the problem of finding  $m=(m^1, m^2) \in (\mathcal{S}^{\rho^1})^{N_T+1} \times (\mathcal{S}^{\rho^2})^{N_T+1}$ such that, for $\ell=1$, $2$, we have 
\be\label{scheme_nonlinear_stochastic_case_multiple_populations}\ba{rcl}
m_{i,0}^\ell &=& \bar{m}_{0}^\ell (E_{i}^\ell)  \hspace{0.4cm} \forall \; i\in \ZZ^{d_\ell},  \\[6pt]
m_{i, k+1}^\ell &=& \frac{1}{2r_\ell}\sum\limits_{p=1}^{r_\ell}\sum\limits_{j \in \ZZ^{d_\ell}} \left[\beta_{i}^\ell ( \Phi_{j,k}^{\ell,p,+}[m])+\beta_{i}^\ell(\Phi_{j,k}^{\ell,p,-}[m])\right] m_{j,k}^\ell \hspace{0.4cm} \forall \; i\in \ZZ^{d_\ell}, \; \; k=0,\hdots, N_T-1, 
\ea \ee  
where
$$\ba{rcl}
\Phi_{j,k}^{\ell,p,+}[m]  &:=&  x_{j}^\ell + h b^\ell(m,x_{j}^\ell,t_{k})+ \sqrt{r_\ell h }\sigma_{p}(m,x_{j}^{\ell},t_{k}), \\[8pt]
\Phi_{j,k}^{\ell,p,-}[m]  &:=&  x_{j}^\ell + h b^\ell(m,x_{j}^\ell,t_{k})-\sqrt{r_\ell h }\sigma_{p}(m, x_{j}^{\ell},t_{k}).
\ea
$$
Arguing exactly as in the proof of Proposition 3.1 in \cite{Carlini_Silva_2017}, the existence of at least one solution $m_{\rho,h}$ is a consequence of {\bf(H)} and Schauder fixed-point theorem. We also point out that the scheme is conservative. Indeed, for $\ell=1,2$ and $k=0,\hdots, N_T$ we have 
$$
\sum_{i \in \ZZ^{d_\ell}} m^\ell_{i,k+1}=\sum_{j \in \ZZ^{d_\ell}}m^{\ell}_{j,k} \frac{1}{2r_\ell} \sum\limits_{p=1}^{r_\ell} \sum_{i \in \ZZ^{d_\ell}}\left[\beta_{i}^\ell ( \Phi_{j,k}^{\ell,p,+}[m])+\beta_{i}^\ell(\Phi_{j,k}^{\ell,p,-}[m])\right]=\sum_{j \in \ZZ^{d_\ell}}m^{\ell}_{j,k}=1,
$$
where the last equality follows from $\sum_{j \in \ZZ^{d_\ell}}m^{\ell}_{j,0}=1$. 
\begin{remark}\label{numerical_inefficiency_one_fp}  $\;$ \smallskip\\
{\rm(i)} As we discussed at the end of the previous section, we could approximate a solution to $(FPK)$ by first approximating a solution of $(FPK')$ and then taking its marginals with respect to $\RR^{d_1}$ and $\RR^{d_2}$. The problem of this approach is that if we use scheme \eqref{scheme_nonlinear_stochastic_case_II} in order to approximate $(FPK')$, then we should consider a discretization of $\RR^{d_1+d_2}$ instead of discretizing $\RR^{d_1}$ and $\RR^{d_2}$ separately {\rm(}as we do with scheme \eqref{scheme_nonlinear_stochastic_case_multiple_populations}{\rm)}, which affects enormously the computational time.  Of course, in our numerical experiments we must consider bounded space grids {\rm(}see the next section{\rm)}, but the same difficulty arises. \smallskip\\
{\rm(ii)} Note that if for each $(x,t) \in \RR^{d_\ell}\times [0,T]$ {\rm(}$\ell=1,2${\rm)} the functions 
$$\ba{l}C([0,T]; \P_1(\RR^{d_\ell}))^2  \ni (m^1,m^2) \mapsto b^{\ell}(m^1,m^2,x,t)\in \RR^{d_\ell} \\[6pt]
\mbox{{\rm and} } \; C([0,T]; \P_1(\RR^{d_\ell}))^2  \ni (m^1,m^2) \mapsto \sigma^{\ell}(m^1,m^2,x,t) \in \RR^{d_\ell \times r_\ell},\ea$$ 
depend on $\left\{(m_1(s),m_2(s)) \; | \; 0\leq s \leq t\right\}$, then the scheme \eqref{scheme_nonlinear_stochastic_case_multiple_populations} is explicit and, as a consequence, it  admits a unique solution. On the other hand, if  $b^{\ell}(m^1,m^2,x,t)$, or $\sigma^{\ell}(m^1,m^2,x,t)$, depends on values  $(m^1(s), m^2(s))$, for some $s\in [t,T]$, then the scheme is implicit and ad-hoc techniques should be used in order to compute a solution numerically. 
\end{remark}
\subsection{Convergence} In this section we analyse the limit behaviour of solutions $(m^1_n, m^2_n)$ to \eqref{scheme_nonlinear_stochastic_case_multiple_populations} with steps $\rho^1_n$, $\rho^2_n$ and $h_n:= 1/N_T^n$ tending to zero as $n\to \infty$. We work with the extensions, defined through \eqref{linear_interpolation_extension_of_simplex},  of $m^1_n$ and $m^2_n$ to  $C([0,T];\P_1(\RR^{d_1}))$ and  $C([0,T];\P_1(\RR^{d_2}))$, respectively. 

 The first important remark is that, as the next result shows, the sequence $(m^1_n, m^2_n)$ is equicontinuous in $C([0,T];\P_1(\RR^{d_1}))\times C([0,T];\P_1(\RR^{d_2}))$ (see \eqref{equicontinuity_of_sequence_of_solutions}) and, for each $t\in [0,T]$, we have that $(m^1_n(t), m^2_n(t))$ belongs to a fixed relatively compact subset of $\P_1(\RR^{d_1})\times \P_1(\RR^{d_2})$ (see \eqref{second_order_bounded_moments} and \eqref{condition_for_compactness}).
\begin{proposition}\label{compactness_of_the_sequence_of_solutions} Suppose that {\bf(H)} holds true and that, as $n\to \infty$, $\rho_1^n + \rho_2^n =O(h_n^2)$. Then, there exists a constant $C>0$ (independent of $n$) such that 
\begin{align} 
d_{1}(m_n^1(t), m_n^1(s)) +d_{1}(m_n^2(t), m_n^2(s))  &\leq  C \sqrt{|t-s|} \hspace{0.6cm} \forall \; t, \; s\in [0,T], \label{equicontinuity_of_sequence_of_solutions}\\
 \int_{\RR^{d_1}}|x|^2 \dd m_n^1(t)(x) + \int_{\RR^{d_2}}|x|^2 \dd m_n^2(t)(x)   &\leq  C  \hspace{0.6cm} \forall \; t \in [0,T]. \label{second_order_bounded_moments}
\end{align}
\end{proposition}
The proofs of \eqref{equicontinuity_of_sequence_of_solutions} and \eqref{second_order_bounded_moments} are analogous to the proofs of \cite[Proposition 4.1]{Carlini_Silva_2017} and \cite[Proposition 4.2]{Carlini_Silva_2017}, respectively, and will therefore be omitted. As a consequence of the previous result and the Arzel\`a-Ascoli theorem, there exists at least one limit point $(m^1,m^2) \in C([0,T];\P_1(\RR^{d_1}))\times C([0,T];\P_1(\RR^{d_2}))$ of $(m^1_n,m^2_n)$. 
In order to prove that any limit point of $(m^1_n, m^2_n)$ solves $(FPK)$, we will assume in addition \smallskip\\
{\bf(Lip)} For $\ell=1$, $2$,  $\mu \in \M$  and compact set $K_\ell  \subseteq \RR^{d_\ell}$,   there exists  $C_{\ell}=C(\mu,K_{\ell})>0$ such that
$$
|b^\ell(\mu,y,t)-b^\ell(\mu, x,t)| +|\sigma^\ell(\mu,y,t)-\sigma^\ell(\mu, x,t)| \leq C_\ell |y-x|  \hspace{0.2cm} \forall  \;  x, \; y \in  K_\ell, \; \;   t\in [0,T].
$$ \vspace{0.000000000001cm}
\begin{theorem}\label{convergence_result} Suppose that {\bf(H)}-{\bf(Lip)}  hold true and that, as $n\to \infty$, $\rho_1^n + \rho_2^n =o(h_n^2)$. Then, every limit point $(m^1,m^2)$ of $(m^1_n,m^2_n)$ {\rm(}there exists at least one{\rm)} solves $(FPK)$.
\end{theorem}
\begin{proof} The proof is analogous to the proof of \cite[Theorem 4.1]{Carlini_Silva_2017} and so we only sketch the main steps. Let $(m^1,m^2)  \in C([0,T];\P_1(\RR^{d_1}))\times C([0,T];\P_1(\RR^{d_2}))$ be a limit point of $(m^1_n, m^2_n)$ and consider a subsequence, still  labelled by $n$, such that $(m^1_n, m^2_n) \to (m^1,m^2)$ as $n\to \infty$. Then, for any $t\in [0,T]$ and $\varphi \in C_{0}^{\infty}(\RR^{d_\ell})$ ($\ell=1$, $2$) we have
\be\label{sumatelescopica}
\int_{\RR^{d_\ell}} \varphi(x) \dd m^{\ell}_{n}(t_{n'})(x)= \int_{\RR^{d_\ell}} \varphi(x) \dd m^{\ell}_n(0)(x)+  \sum_{k=0}^{n'-1} \int_{\RR^{d_\ell}} \varphi(x) \dd \left[m^{\ell}_n(t_{k+1})-m_n^{\ell}(t_{k})\right](x),
\ee
where $n' \in \{0, \hdots, N_{T}^{n}\}$ is such that $t_{n'}=n'h_n \to t$. Using \eqref{scheme_nonlinear_stochastic_case_multiple_populations}, 
we obtain that 
$$\ba{rcl}
\int_{\RR^{d_\ell}} \varphi(x) \dd  m^{\ell}_n(t_{k+1})(x)&=& \sum_{i \in \ZZ^{d_\ell}} \varphi(x_i)m^{\ell}_{k+1,i}\\[10pt]
\; &=&\sum_{i \in \ZZ^{d_\ell}} \varphi(x_i)\frac{1}{2r_\ell}\sum\limits_{p=1}^{r_\ell}\sum\limits_{j \in \ZZ^{d_\ell}} \left[\beta_{i}^\ell ( \Phi_{j,k}^{\ell,p,+}[m_n])+\beta_{i}^\ell(\Phi_{j,k}^{\ell,p,-}[m_n])\right] m_{j,k}^\ell\\[10pt]
\: &=& \sum\limits_{j \in \ZZ^{d_\ell}} \frac{m_{j,k}^\ell}{2r_\ell}\sum\limits_{p=1}^{r_\ell} \left[I[\varphi](\Phi_{j,k}^{\ell,p,+}[m_n])+I[\varphi](\Phi_{j,k}^{\ell,p,-}[m_n])\right]\\[10pt]
\: &=&\sum\limits_{j \in \ZZ^{d_\ell}} \frac{m_{j,k}^\ell}{2r_\ell}\sum\limits_{p=1}^{r_\ell} \left[ \varphi\left(\Phi_{j,k}^{\ell,p,+}[m_n]\right)+ \varphi \left(\Phi_{j,k}^{\ell,p,-}[m_n]\right)\right] + O((\rho_n^\ell)^2),
\ea
$$
where in the last equality we have used that $\sup_{x\in \RR^{d_\ell}} |I[\varphi](x)-\varphi(x)|= O((\rho_n^\ell)^2)$. By a Taylor expansion, we obtain 
$$\ba{rcl}
\varphi\left(\Phi_{j,k}^{\ell,p,+}[m_n]\right)+ \varphi \left(\Phi_{j,k}^{\ell,p,-}[m_n]\right)&=& 2\phi(x_j) + 2 h_n\nabla \varphi(x_j) \cdot b^\ell(m_n^1,m_n^2, x_j,t_k)\\[10pt]
\; & \; & + r_{\ell}h_n \sum\limits_{1\leq i', j' \leq d_\ell} \partial_{x_{i'},x_{i'}}\varphi(x_j) \sigma_{i',p}^\ell \sigma_{j',p}^\ell\\[12pt]
\; & \; & + O(h_n^2),
\ea
$$
where we have omitted the dependence of $\sigma_{i',p}^\ell$ and $\sigma_{j',p}^\ell$ on $(m_n^1,m_n^2, x_j,t_k)$. This implies that 
$$\ba{rcl}
\frac{1}{2r_\ell}\sum\limits_{p=1}^{r_\ell} \left[ \varphi\left(\Phi_{j,k}^{\ell,p,+}[m_n]\right)+ \varphi \left(\Phi_{j,k}^{\ell,p,-}[m_n]\right)\right]&=&\phi(x_j)+ h_n\nabla \varphi(x_j) \cdot b^\ell(m_n^1,m_n^2, x_j,t_k) \\[6pt]
\; & \; & + \frac{h_n}{2}  \sum\limits_{1\leq i', j' \leq d_\ell} \partial_{x_{i'},x_{i'}}\varphi(x_j)a_{i',j'}^\ell(m_n^1,m_n^2, x_j,t_k) \\[13pt]
\; & \; & + O(h_n^2).
\ea$$
Thus, using \eqref{sumatelescopica}, we obtain 
$$\ba{l}
\int_{\RR^{d_\ell}} \varphi(x) \dd m^{\ell}_{n}(t_{n'})(x)= \int_{\RR^{d_\ell}} \varphi(x) \dd m^{\ell}_n(0)(x)\\[10pt]
+ h_n\sum_{k=0}^{n'-1} \int_{\RR^{d_\ell}}\left[\nabla \varphi(x) \cdot b^\ell(m_n^1,m_n^2, x,t_k)+\frac{h_n}{2}  \sum\limits_{1\leq i, j\leq d_\ell} \partial_{x_{i},x_{i}}\varphi(x)a_{i,j}^\ell(m_n^1,m_n^2, x,t_k)\right]\dd m_n^\ell(t_k) \\[14pt]
+O\left(h_n + \frac{(\rho_n^\ell)^2}{h_n}\right).
\ea
$$
Finally, using that $m_n^\ell \to m^\ell  \in C([0,T]; \P_1(\RR^{d_\ell}))$ by {\bf(H)} we have that $b^\ell(m_n^1,m_n^2, \cdot,\cdot) \to b^\ell(m^1,m^2, \cdot,\cdot)$ and $a_{i,j}^\ell(m_n^1,m_n^2, \cdot,\cdot) \to a_{i,j}^\ell(m^1,m^2, \cdot,\cdot)$ uniformly in $\mbox{supp}(\varphi) \times [0,T]$ (where $\mbox{supp}(\varphi)$ denotes the support of $\varphi$, which is a compact set). Using this fact and assumption {\bf(Lip)}, we can argue in the same manner than in \cite[Theorem 4.1]{Carlini_Silva_2017} and   pass to the limit in the expression above to obtain  that $m^\ell$ satisfies \eqref{solution_FP}. The result follows. 
\end{proof}
\begin{remark}
As in \cite[Theorem 4.2]{Carlini_Silva_2017}, we can get rid of assumption {\bf(Lip)} at the price of regularizing by convolution the coefficients $b^\ell$ and $\sigma^\ell$ and considering the associated scheme with the regularized coefficients.
\end{remark}
In practice we have not always access to the coefficients $b^\ell$ and $a_{i,j}^\ell$ and they have to be approximated. As we will see in the next section, this is the case of multi-population MFGs systems. Consider a sequence of space steps $\rho^1_n$, $\rho^2_n$ and a sequence of time steps $h_n$ satisfying the assumptions of the previous result. Assume that for each $n$ we have  
$$\ba{rcl} b^\ell_n: C([0,T]; \P_1(\RR^{d_\ell})) \times \RR^{d_\ell}\times [0,T] &\to& \RR^{d_\ell}, \\[8pt]
 \sigma^\ell_n: C([0,T]; \P_1(\RR^{d_\ell})) \times \RR^{d_\ell}\times [0,T] &\to& \RR^{d_\ell \times r_\ell},
\ea
$$such that:\smallskip\\
{\bf(H')}
{\rm(i)} for each fixed $t\in [0,T]$, the mappings $b^\ell_n(\cdot, \cdot, t)$ and $\sigma^\ell_n(\cdot, \cdot, t)$ are continuous. \smallskip\\
{\rm(ii)}  the growth condition \eqref{linear_growth_1} holds for a constant $C>0$ independent of $n$. \smallskip\\
{\rm(iii)} for any sequence $\mu_n \in  C([0,T]; \P_1(\RR^{d_\ell}))$ and $\mu \in C([0,T]; \P_1(\RR^{d_\ell}))$ satisfying that $\mu_n \to \mu$ we have  
$$
b^\ell_n(\mu_n, \cdot, \cdot) \to b^\ell(\mu, \cdot, \cdot), \hspace{0.5cm } \sigma^\ell_n(\mu_n, \cdot, \cdot) \to \sigma^\ell(\mu, \cdot, \cdot)
$$
uniformly on compact subsets of $\RR^{d_\ell} \times [0,T]$.  \medskip

Consider the scheme \eqref{scheme_nonlinear_stochastic_case_multiple_populations} constructed with discrete characteristics
$$\ba{rcl}
(\Phi_{j,k}^{\ell,p,+})_n[m]  &:=&  x_{j} + h b_n^\ell(m,x_{j},t_{k})+ \sqrt{r h }(\sigma_n^\ell)_{p}(m,x_{j},t_{k}), \\[8pt]
(\Phi_{j,k}^{\ell, p,-})_n[m]  &: =&  x_{j} + h b_n^\ell(m,x_{j},t_{k})-\sqrt{r h }(\sigma_n^\ell)_p(m, x_{j},t_{k}),
\ea
$$
which, by similar arguments to those in the case of coefficients independent of $n$, admits at least one solution $(m_n^1, m_n^2)$. 
Then, we have the following result, whose proof is analogous to the proof of Theorem \ref{convergence_result}.
\begin{theorem}\label{convergence_with_coefficient_approximation}
Under {\bf(H)}-{\bf(Lip)} and the previous assumptions, the sequence $(m_n^1, m_n^2)$ admits at least one limit point $(m^1, m^2) \in  C([0,T]; \P_1(\RR^{d_1})) \times  C([0,T]; \P_1(\RR^{d_2}))$. Moreover, every such limit point solves $(FPK)$. 
\end{theorem} 
\section{Simulations}\label{simultions}
We show the performance of our scheme  by applying it to approximate the solution of two instances of $(FPK)$ with $M=2$. In the first example we consider a variation of a PDE system treated analytically in \cite{CarLab15} and numerically in \cite{BCL16}, which models the evolution of two interacting species. In our framework,  the drifts $b^1$ and $b^2$ have  non local cross interaction terms and also a term that will approximate a nonlinear diffusion term present in \cite{CarLab15,BCL16}. In the second example, we consider a particular instance of a two population MFG system modelling segregation (see e.g. \cite{AcBarCi17,MR3333058}). As discussed in \cite{Carlini_Silva_2017}, standard MFGs can be seen as a  particular   $(FPK)$ equation with $M=1$, where the drift term $b^1$ satisfies that for each $(x,t)\in \RR^{d_1} \times [0,T]$ the function  $C([0,T]; \P_1(\RR^{d_1})) \ni m \mapsto  b^1(m,x,t)\in \RR^{d_1}$ depends on the values $\{m(s) \; | \; s \in (t,T]\}$. When $M\neq 1$, the situation is similar and hence, as explained in Remark \ref{numerical_inefficiency_one_fp}{\rm(ii)}, the scheme is implicit. 
 

Since the scheme \eqref{scheme_nonlinear_stochastic_case_multiple_populations} is defined on the unbounded space grid $\G_{\rho}$, in our numerical examples we need to change this grid to a bounded one. In order to maintain the total mass constant, we impose homogeneous Neumann boundary conditions  and near the boundary we approximate the  discrete flow by using a projected Euler scheme, as  proposed in  \cite{Costantini98}. The proof of convergence of the modified scheme is postponed to a future work. 

In all tests that  we chose the discretization parameters  $(\rho,h)$ satisfying $h=O(\rho^{3/2})$, which is less restrictive than the classical parabolic CFL condition for explicit finite difference schemes. Larger time step would produce loss of accuracy close to the boundary.  The question on how to modify the scheme  at the boundary maintaining  large time steps will also be addressed in a future work.

 In the examples that we present below, at each time $t\in [0,T]$ the solution $(m^1,m^2)$ is shown to admit a density with respect to the Lebesgue measure.   For each ${\ell}=1,2$  we approximate the density of $m^\ell$ by defining  $\mathbf{m}^{\ell}_{\rho,h}(x,t):=m^{\ell}_{i,k}/\rho^{d_{\ell}}$ if $(x,t)\in E^{\ell}_i\times  [t_k,t_{k+1})$. For fixed $t$,  $\mathbf{m}^{\ell}_{\rho,h}$ is a density which is  uniform on each $E_i^{\ell}$.
\subsection{Interacting species}
We consider a system of two interacting species proposed first in the first order case in \cite{DiFrancescoFagioli13} and then extended in \cite{CarLab15} to the case where a nonlinear diffusion term is also added to the system. The densities  $m^1$ and $m^2$ of the two species  are coupled through the drift by non local terms. The system studied in  \cite{CarLab15} reads 
\be\label{species}
\begin{cases}
	\partial_t  m^1 -\mbox{div}\left( m^1\left(\nabla E'(m^1)+\nabla U_1( m^1,m^2,  x,t ) \right)\right)=0, \\[6pt]	
	\partial_t  m^2 - \mbox{div}\left( m^2 \left(\nabla E'(m^2)+\nabla U_2(m^1,m^2, x,t) \right)\right)=0, \\[6pt]
	m^{1}(\cdot, 0)= m_0^1(\cdot), \hspace{0.5cm} m^{2}(\cdot, 0)= m_0^2(\cdot).
\end{cases}
\ee
In \eqref{species}, $m_0^\ell$ ($\ell=1,2$) represent two absolutely continuous probability measures whose densities are still denoted by $m_0^\ell$. The term  $E(m):=\frac{1}{2}m^3$ corresponds to an internal energy which introduces the nonlinear diffusion term $-\mbox{div}( m^\ell(\nabla E'(m^\ell))=-\Delta (m^\ell)^3$ in \eqref{species}. It is assumed that $\int_{\RR^d}(m_{0}^{\ell}(x))^3 \dd x <+\infty$ for $\ell=1$, $2$.
 The potentials $U_1$, $U_2: C([0,T]; \P_{1}(\RR^d))^2 \times \RR^d \times [0,T] \to \RR$ are cross interactions terms and they are given by convolution with smooth functions
 $$\ba{rcl}
 U_1( m^1,m^2,x,t) &=&W_{11}\ast  \left[m^1(t)\right](x)+W_{21}\ast  [m^2(t)](x), \\[6pt] 
  U_2(m^1,m^2,x,t) &=&W_{12}\ast  [m^1(t)](x)+W_{22}\ast [m^2(t)](x),
\ea
 $$
where $\ast$ denotes the space convolution and  $W_{11}(x)=W_{21}(x)=W_{22}(x):=\frac{|x|^2}{2}$, 
 $W_{12}(x):=\frac{-|x|^2}{2}$. With these choices, the drift terms 
\be\label{first_coefficients_W}-\nabla\left(W_{11}\ast m^1(t)\right)(x)= \int_{\RR^2}(y-x)\dd m^1(t)(y), \; \; \; -\nabla \left(W_{22}\ast m^2(t)\right)(x)=\int_{\RR^2}(y-x)\dd m^2(t)(y)\ee
model self-interactions for the first and second species, respectively, whereas the terms
\be\label{second_coefficients_W}- \nabla \left(W_{21}\ast m^2(t)\right)(x)= \int_{\RR^2}(y-x)\dd m^2(t)(y), \; \; \; - \nabla \left(W_{12}\ast m^1(t)\right)(x)= -\int_{\RR^2}(y-x)\dd m^1(t)(y),\ee
model the facts that the first species is attracted by the second one and    that the latter is repelled by first one, respectively. Note that the drift terms in \eqref{first_coefficients_W}-\eqref{second_coefficients_W} do not satisfy {\bf(H)} because the linear growth is not uniform w.r.t. $m^{\ell}$. This can be easily fixed by considering suitable compactly supported $C^\infty$ approximations of the function $y-x$. In our simulations, we work on a bounded domain  and so we work directly with the coefficients \eqref{first_coefficients_W}-\eqref{second_coefficients_W}. It is easy to see that that these drift terms satisfy {\bf(Lip)}.

Existence and uniqueness results of weak solutions to \eqref{species} has been proved in \cite{DiFrancescoFagioli13} when  $E_1=E_2=0$.   In the diffusive case, existence of at least one weak solution, which is absolutely continuous w.r.t. the Lebesgue measure, has been proved  in \cite{CarLab15}. We refer the reader to \cite{BCL16} for the numerical resolution of \eqref{species} by the so-called JKO scheme combined with the augmented Lagrangian method.

Since under {\bf(H)} the coefficients should be continuous with respect  to the weak convergence of probability measures, we need to regularize the local term $E'(m)= \frac{3}{2}m^2$. We do this by convolution. More precisely, given a regularization parameter $\delta>0$ we define $E'_\delta: C([0,T]; \P_1(\RR^d)))\times \RR^d \times [0,T] \to \RR$ as
$$E'_{\delta}(m,x,t):= \frac{3}{2}(m(t)\ast \phi_{\delta}(x))^2,$$
where   $\phi_{\delta}(x)=\sqrt{2\pi}\delta \exp{(-|x|^2/(2 \delta^2))}.$ 
We then consider the following variation of \eqref{species}:
\be\label{modelreg}
\begin{cases}
	\partial_t  m^1 -\mbox{div}( m^1(\nabla E'_{\delta}(m^1)+\nabla U_1(m^1,m^2) ))=0, \\[6pt]	
	\partial_t  m^2 - \mbox{div}( m^2(\nabla E'_{\delta}(m^2)+\nabla U_2(m^1,m^2) ))=0, \\[6pt]
	m^{1}(\cdot, 0)= \bar{m}_0^1(\cdot), \hspace{0.5cm} m^{2}(\cdot, 0)= \bar{m}_0^2(\cdot),
\end{cases}
\ee
which satisfies {\bf(H)}, with the suitable modifications of \eqref{first_coefficients_W}-\eqref{second_coefficients_W}.

\subsubsection{ Numerical test } \label{testspecies}  
We numerically solve  system \eqref{modelreg} with $d_1=d_2=2$ on a domain $\Omega\times [0,T]=[-1,1]\times [-1,1]\times [0,5]$, 
with homogeneous Neumann boundary conditions, $\delta=0.02$ and initial conditions 
$$
m^1(x,0)=\frac{\nu_1(x)}{\bar \nu_1}  \hspace{0.6cm} \mbox{and}  \hspace{0.6cm} m^2(x,0)=\frac{\nu_2(x)}{\bar \nu_2},
$$
where 
$$\ba{rcl}\nu_1(x_1,x_2)&:=&\left[0.2- (x_1-0.5)^2 -\frac{(x_2+0.5)^2}{2} \right]_{+}^2,\\[10pt]
 \nu_2(x_1,x_2)&:=&\left[0.2- (x_1+0.5)^2 - \frac{(x_2-0.5)^2}{2}\right]_{+}^2,
\ea
$$
and, for $a\in \RR$, $a^+:= \max\{0,a\}$, and $\bar \nu_1$, $\bar \nu_2$ are two positive constants such that 
$$\int_\Omega m^1(x,0)\dd x=\int_\Omega m^2(x,0) \dd x=1.$$
 
In Figure \ref{Test3} we display the evolution of the two densities at  the times $t=0$, $1$, $2$, $3$, $4$, $5$ computed  with $\rho=2e^{-2} $ and $h=\frac{1}{3}\rho^{3/2}$. 
The first plot on the top left shows the initial configurations: $\mathbf{m}_{\rho,h}^1$ is represented by the density located on the bottom right and $\mathbf{m}_{\rho,h}^2$ by the density located on the top left of the numerical domain.  As time evolves, we observe the density $\mathbf{m}_{\rho,h}^1$ moving towards the density $\mathbf{m}_{\rho,h}^2$, which is  instead repelled by $\mathbf{m}_{\rho,h}^1$. Due to the presence of Neumann boundary conditions, $\mathbf{m}_{\rho,h}^2$ get  finally captured in the upper left corner of the domain.
We can also observe the effect of the regularization of the nonlinear diffusion terms along with  the effect of the attraction potential $W_{11}$: the numerical support of the density  $\mathbf{m}_{\rho,h}^1$   takes a circular  shape. 
In Figure \ref{Test3b}, we show a 3D view of the initial configuration (left) and the final configurations of  $\mathbf{m}_{\rho,h}^1$(center) and $\mathbf{m}_{\rho,h}^2$(right).
\begin{figure}[ht!]
\begin{center}
\includegraphics[width=4cm]{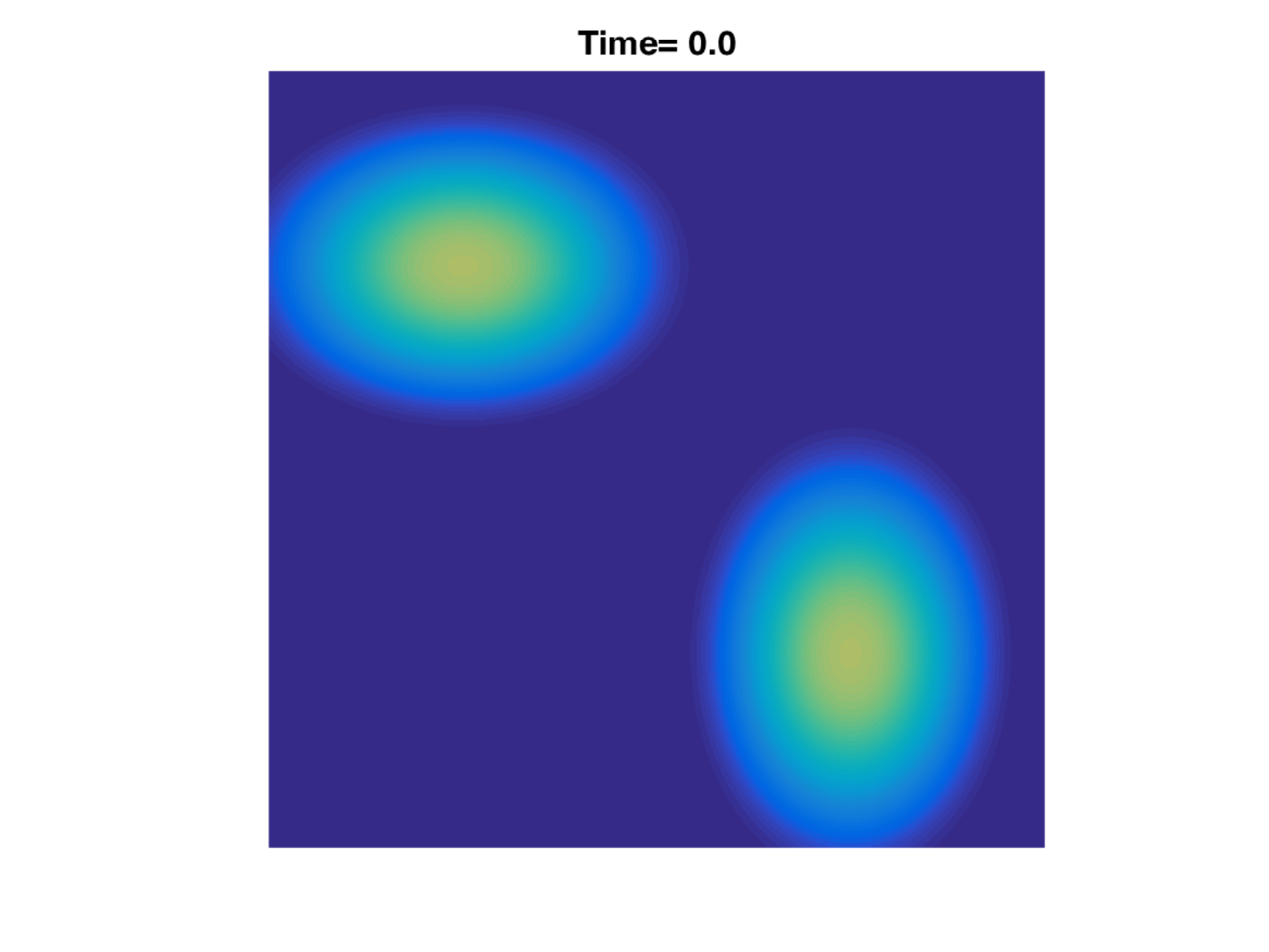}
\includegraphics[width=4cm]{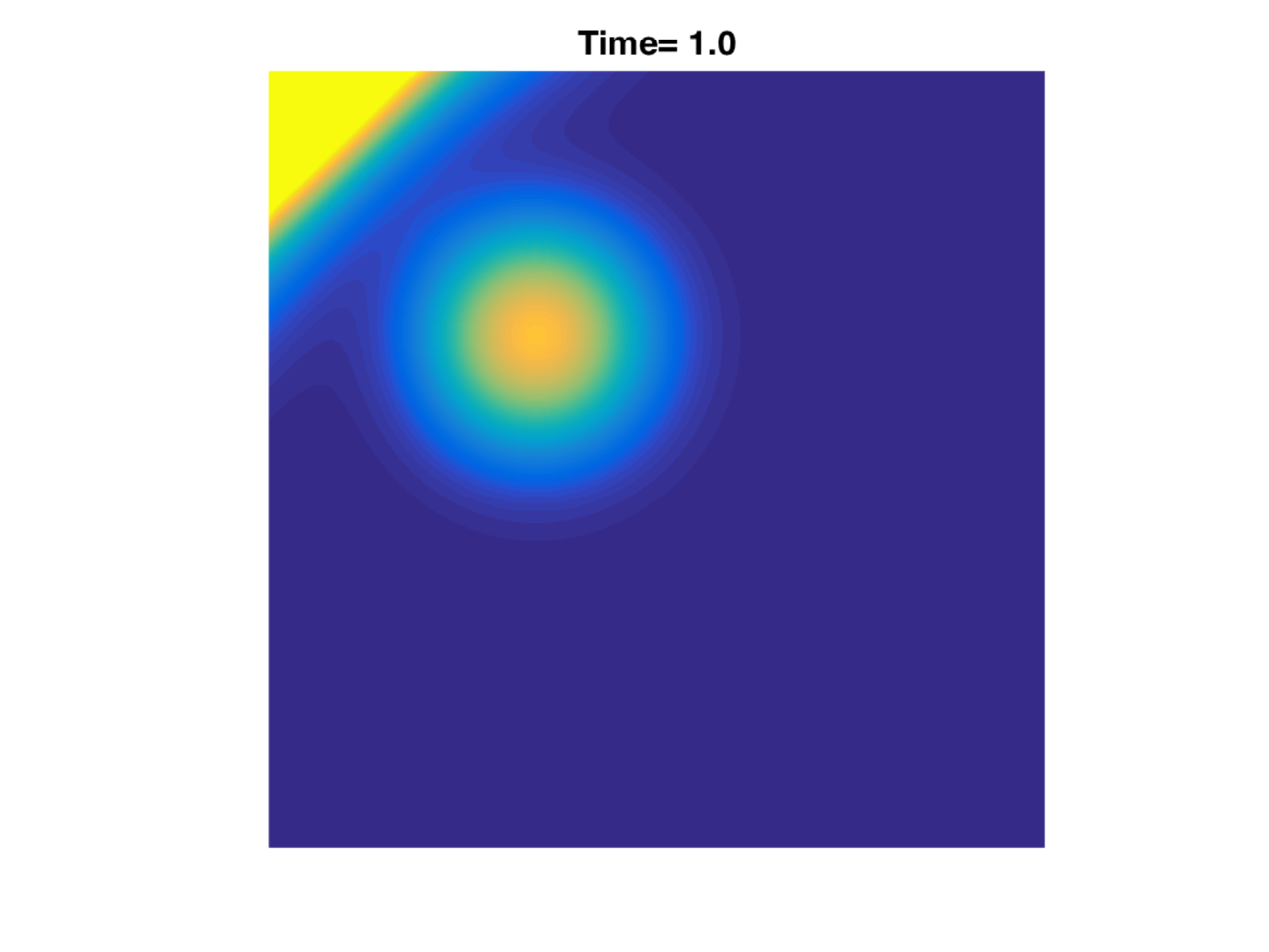}
\includegraphics[width=4cm]{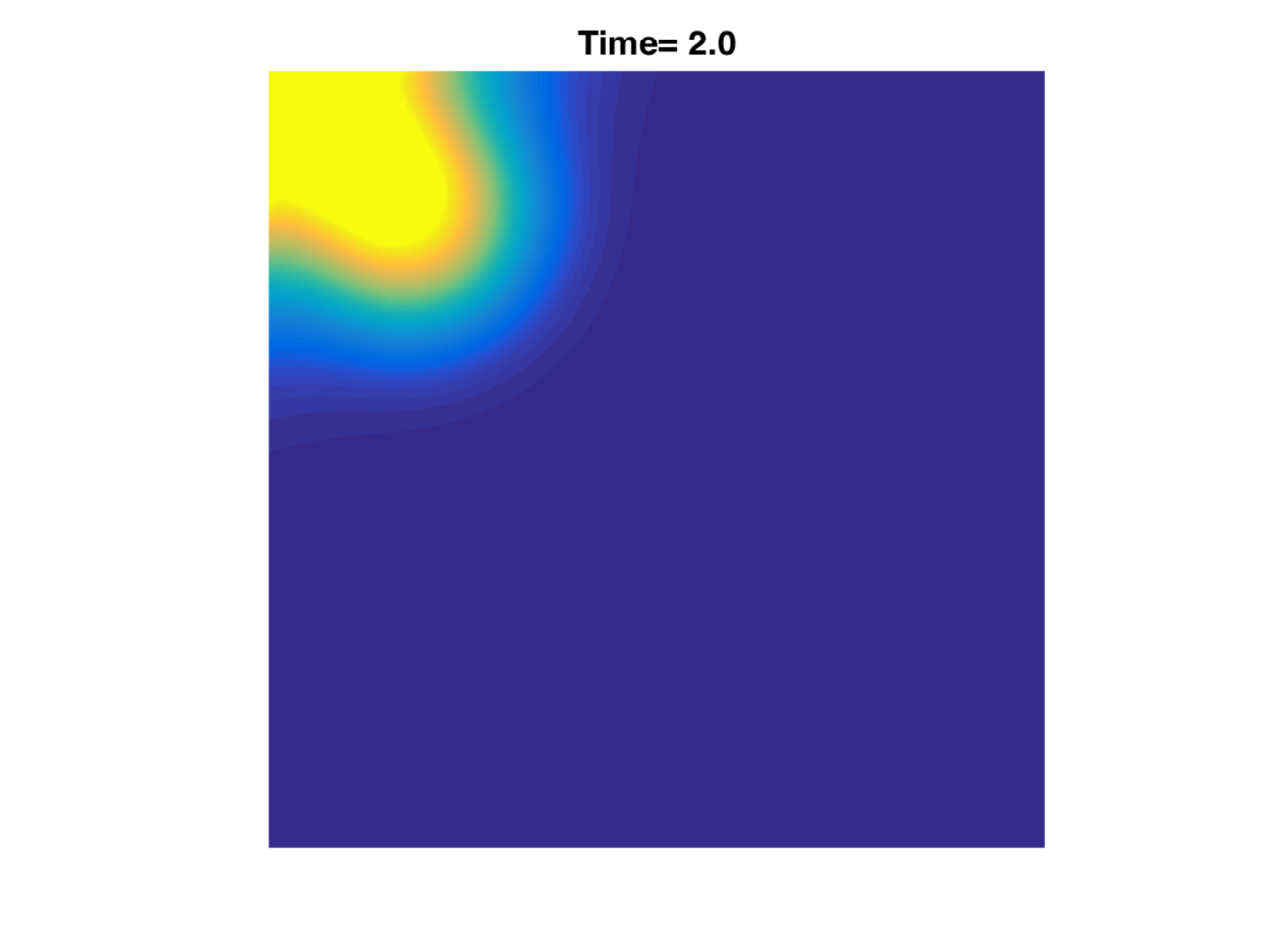}\\
\includegraphics[width=4cm]{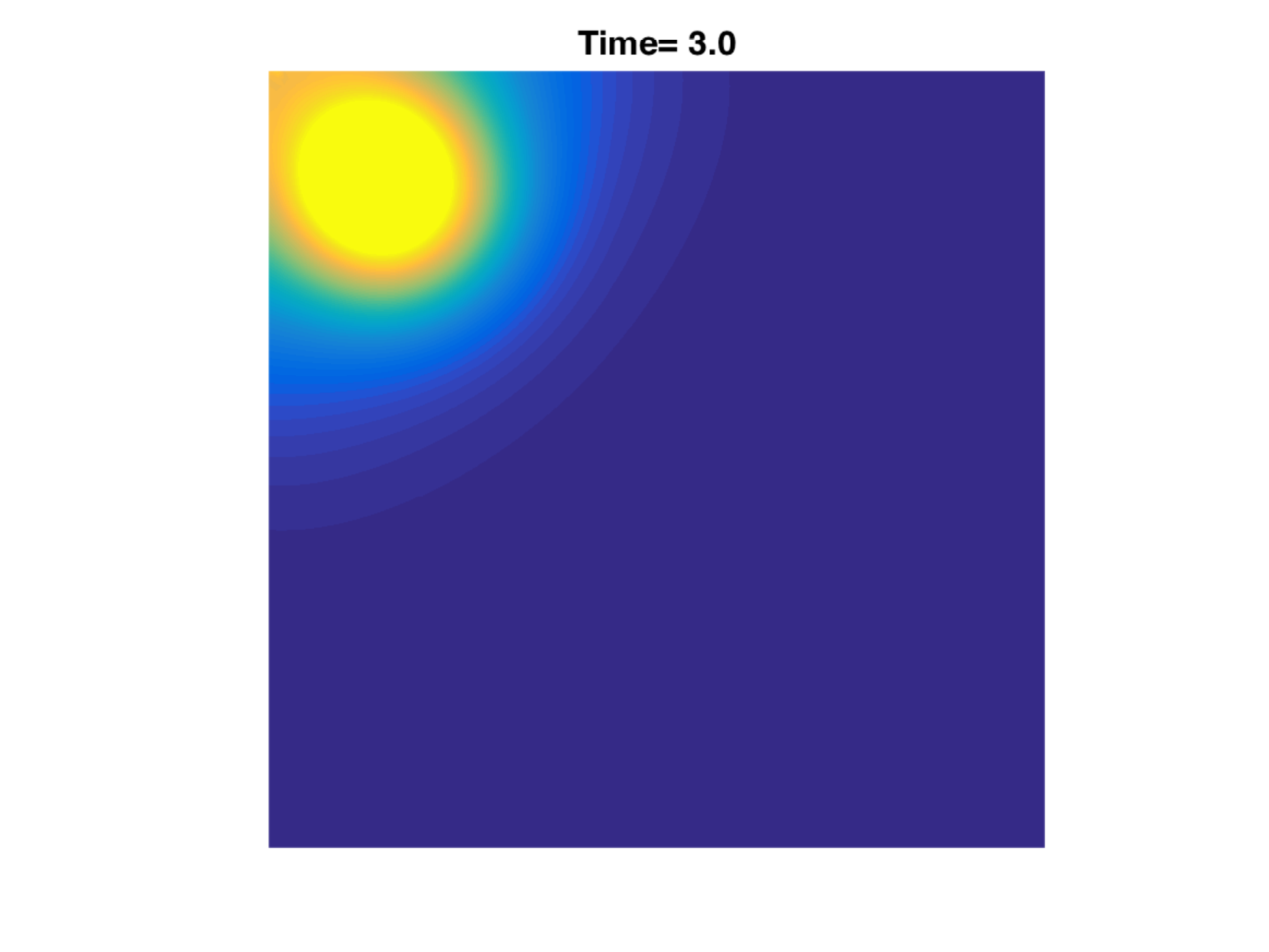}
\includegraphics[width=4cm]{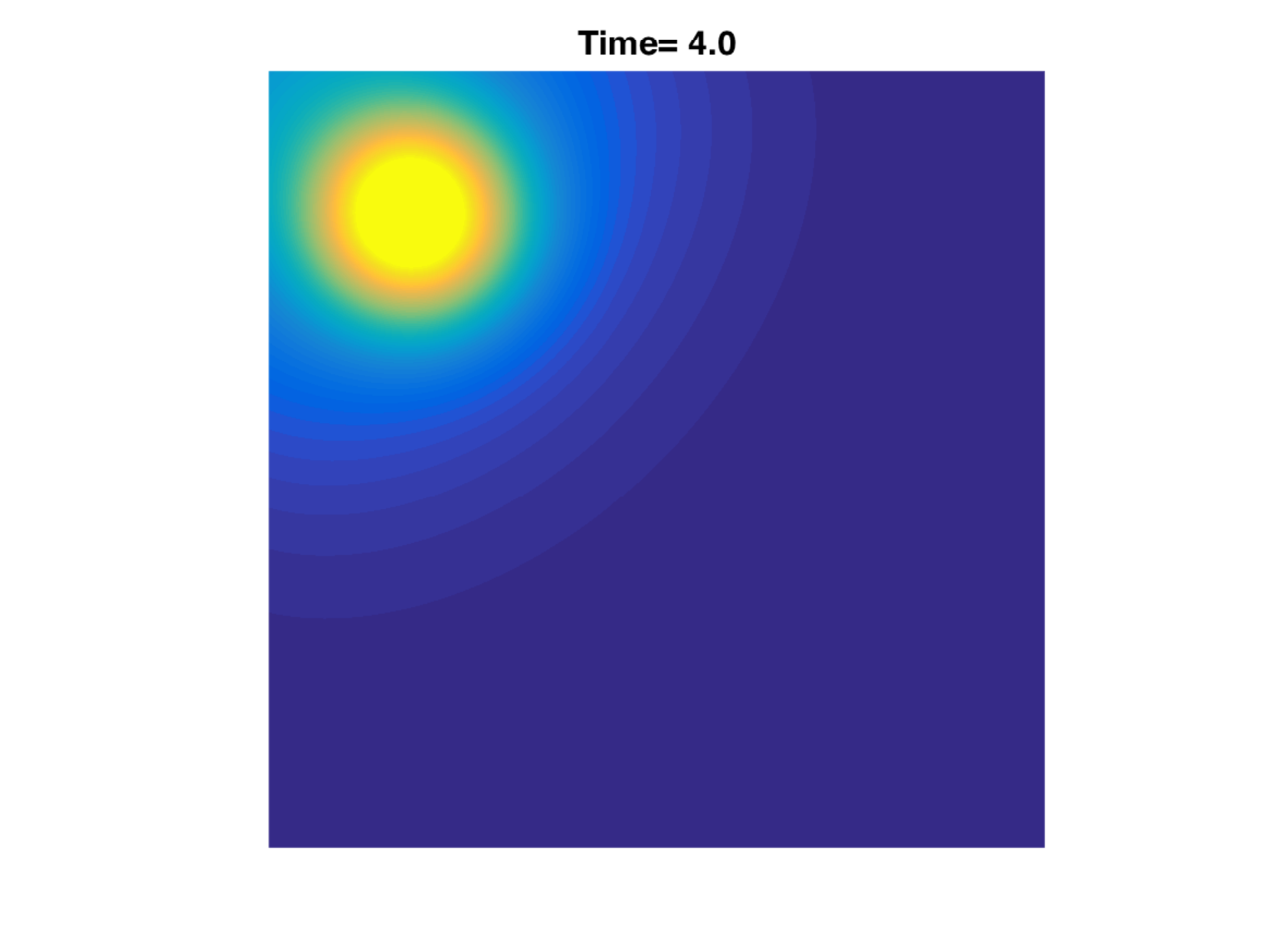}
\includegraphics[width=4cm]{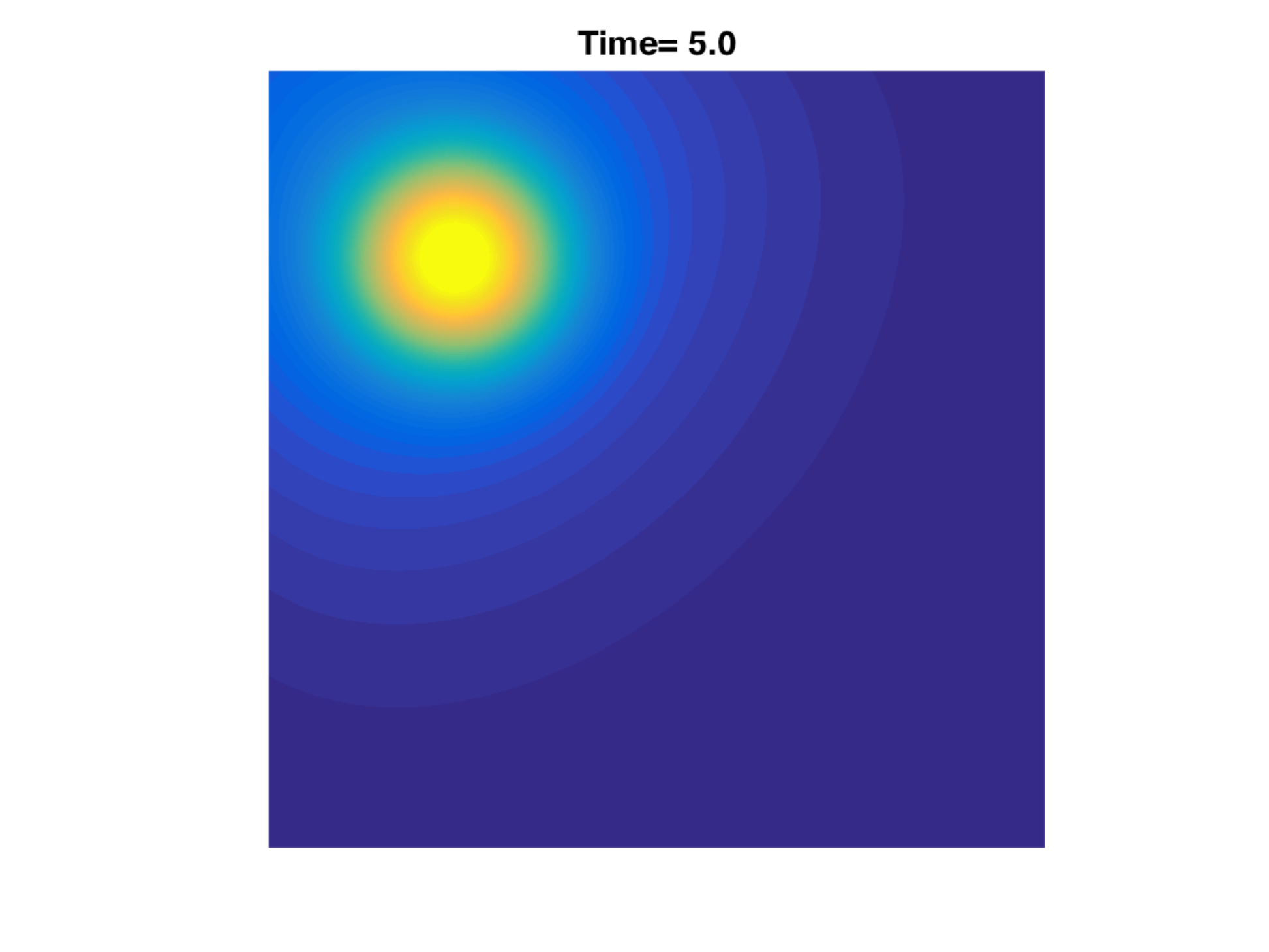}
{\caption{Evolution of  the two densities $\mathbf{m}_{\rho,h}^1$ and $\mathbf{m}_{\rho,h}^2$  at the times  $t=0$, $1$, $2$, $3$, $4$, $5$.
\label{Test3}}}
\end{center}
\end{figure}
\begin{figure}[ht!]
\begin{center}
\includegraphics[width=4cm]{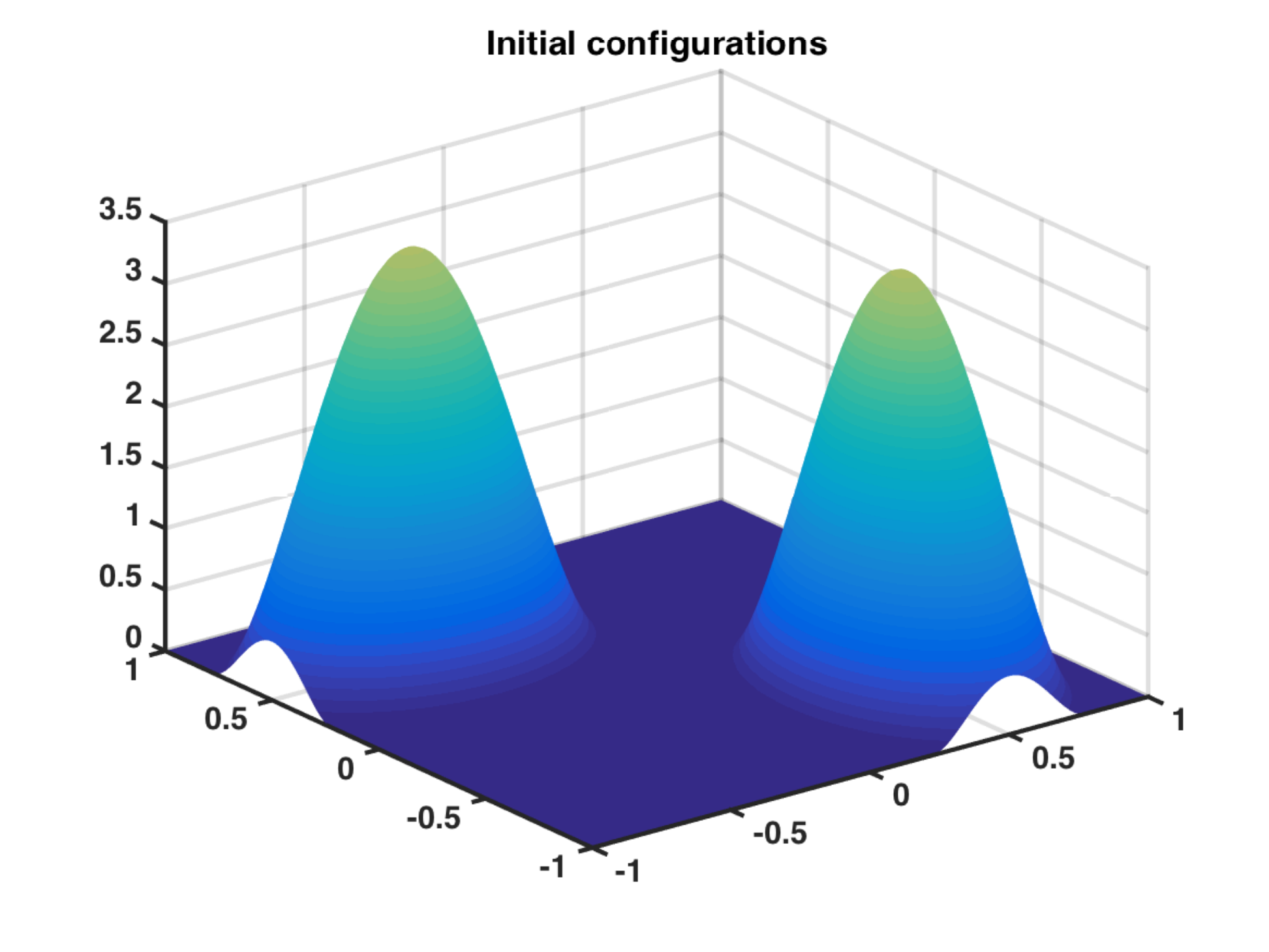}
\includegraphics[width=4cm]{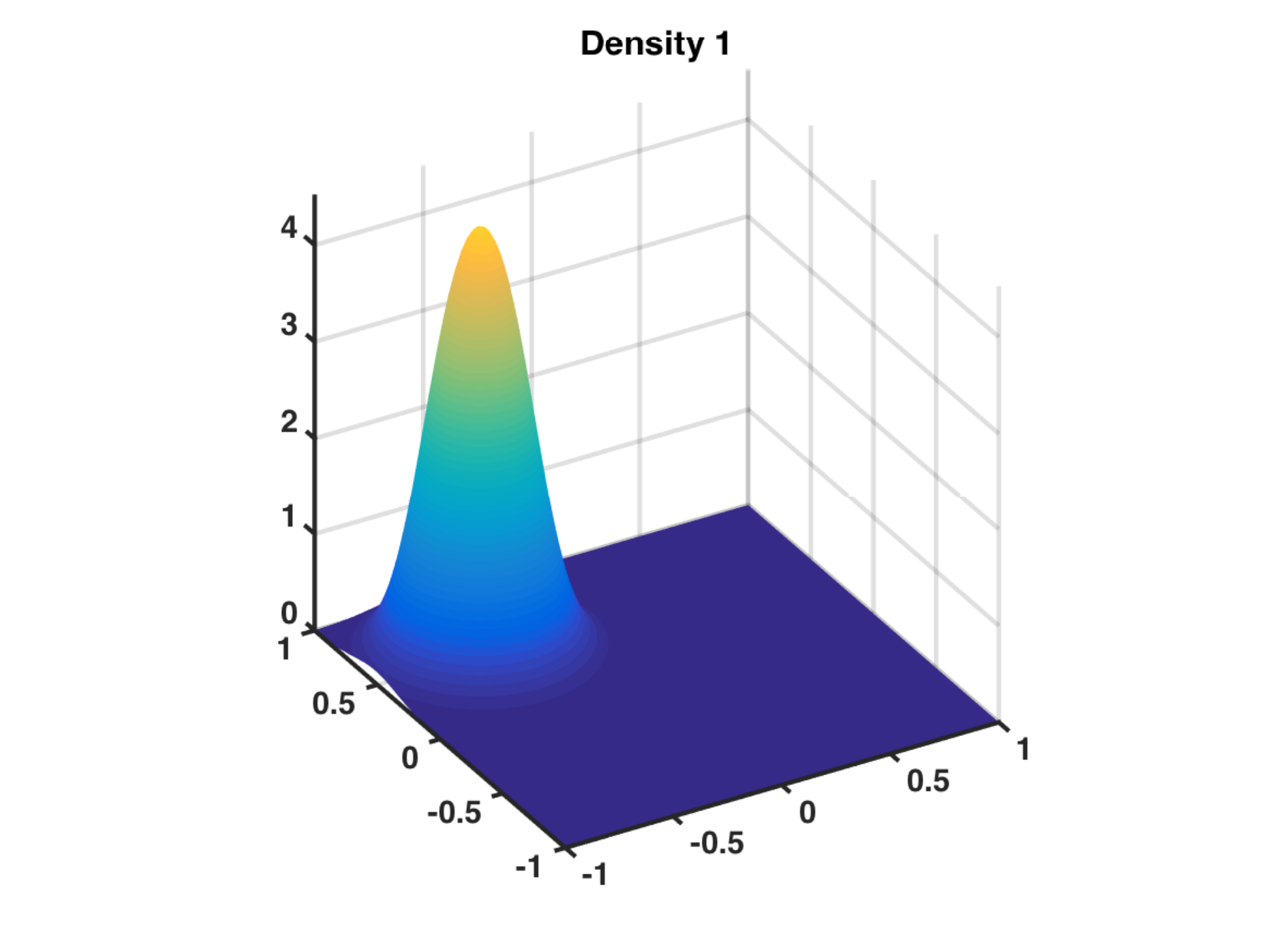}\includegraphics[width=4cm]{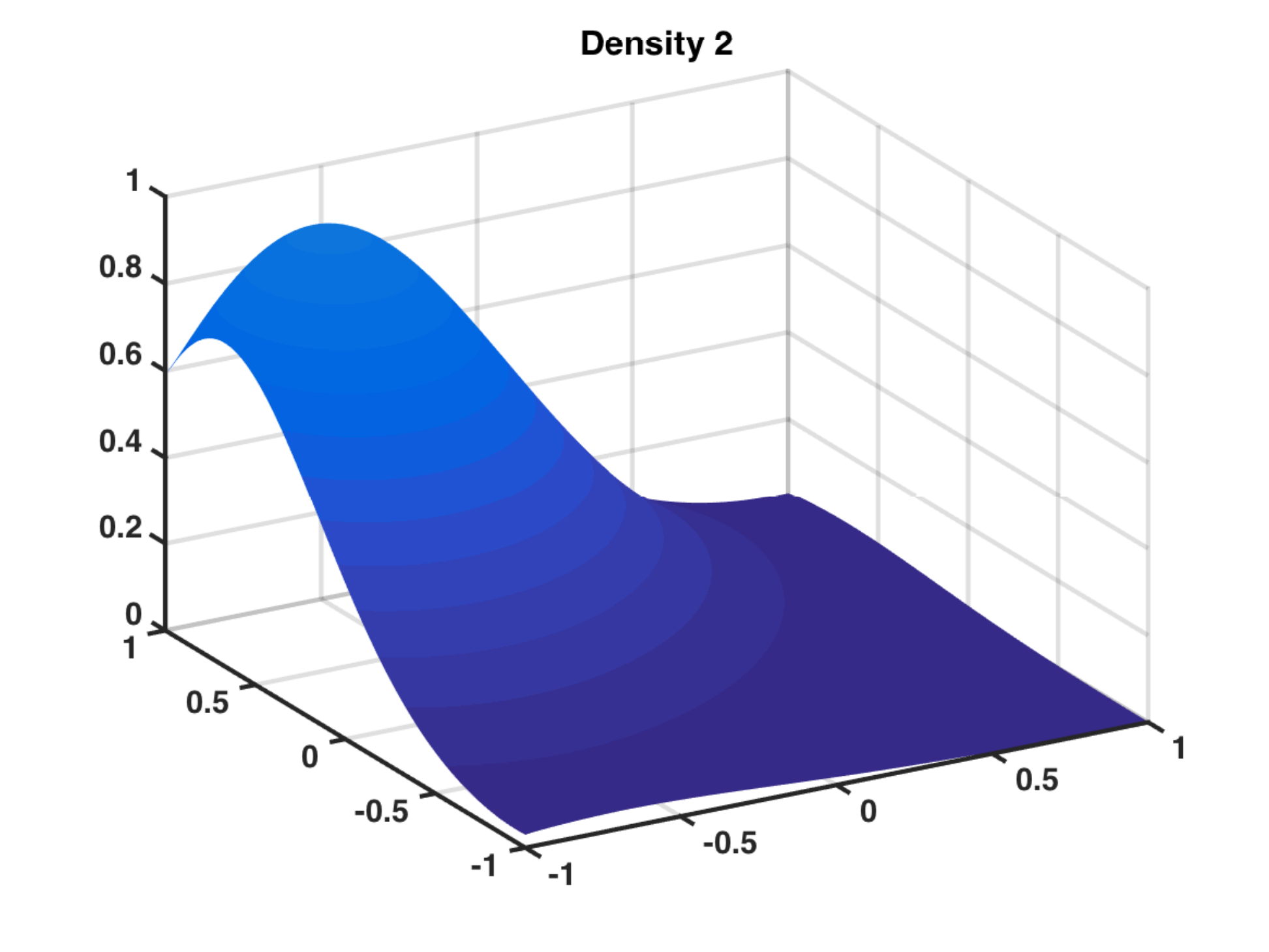}
{\caption{3D view of the initial configuration (left), of the final configuration of  $\mathbf{m}_{\rho,h}^1$ (center) and $\mathbf{m}_{\rho,h}^2$ (right).
\label{Test3b}}}
\end{center}
\end{figure}

\subsection{Two populations Mean Field Games}
In this section, we consider the  following MFG system
\be\label{shelling} \tag{MFG}
\left\{\ba{rl}
	-\partial_t  v^1 -\nu \Delta v^1 +\half |\nabla v^1|^2&=V(m^1,m^2), \\[6pt]	
	-\partial_t  v^2 -\nu \Delta v^2 +\half |\nabla v^2|^2&=V(m^2,m^1), \\[6pt]	
	v^1(\cdot, T)= 0, & \;  v^2(\cdot, T)= 0, \\[6pt]
	\partial_t  m^1 -\nu \Delta m^1 -\mbox{div}( \nabla v^1 m^1)&=0, \\[6pt]	
	\partial_t  m^2 -\nu \Delta m^2 -\mbox{div}( \nabla v^2 m^2)&=0, \\[6pt]
	m^1(\cdot,0)= \bar{m}_0^1(\cdot),  & \; m^2(\cdot,0)= \bar{m}_0^2(\cdot).
\ea \right.
\ee
In the system above, $\nu \geq 0$, $\bar{m}_0^1$, $\bar{m}_0^2 \in L^{\infty}(\RR^d)$ ($d\in \NN \setminus \{0\}$) are densities with compact support and the local coupling term $V: \RR \times \RR  \to \RR$ is given by
\be\label{v_non_regularized}
V(m^1,m^2)=\left( \frac{m^1}{m^1+m^2}-0.7\right)^-+ (m^1+m^2-8)^+, 
\ee
where, for $a\in \RR$, we set $a^-:= a^+ -a$.  This system has been proposed in \cite{AcBarCi17} and models  interactions between two populations with xenophobia  and aversion to overcrowded regions effects.  As in the previous example, we need to regularize the local coupling term $V$ in order to obtain a  function that is continuous with respect to the weak convergence of probability measures. We proceed as in \cite[Section 6.2.1]{AcBarCi17}. Given $\eta$, $\delta>0$, we define $V_{\eta,\delta}: C([0,T]; \P_1(\RR^d))^2 \times \RR^d \times [0,T] \to \RR$ as   
$$\ba{rcl}V_{\eta,\delta} (m^1,m^2,x,t)&=&\Psi_{-,\eta}\left(  \frac{m^1(t) \ast \phi_\delta(x)}{m^1(t) \ast \phi_\delta(x) +m^2(t) \ast \phi_\delta(x)+\eta}-0.7\right)  \\[10pt]
\; & \; &+\Psi_{+,\eta}\left(m^1(t) \ast \phi_\delta(x)+ m^2(t) \ast \phi_\delta(x) -8\right),\ea$$ \normalsize
where    
$$\Psi_{-,\eta}(y):=\begin{cases}
-y+\frac{\eta}{2}(e^{\frac{y}{\eta}}-1) &y\leq 0,\\
\frac{\eta}{2}(e^{-\frac{y}{\eta}}-1) &y>0,\\
\end{cases}\quad \; \; \; 
\Psi_{+,\eta}(y):=\begin{cases}
\frac{\eta}{2}(e^{\frac{y}{\eta}}-1) &y\leq 0,\\
y+\frac{\eta}{2}(e^{-\frac{y}{\eta}}-1) &y> 0,\\
\end{cases}
$$
are smooth approximations of $(\cdot)^-$ and $(\cdot)^+$, respectively, and $m^1(t) \ast \phi_\delta(\cdot)$, $m^2(t) \ast \phi_\delta(\cdot)$ are defined as the convolutions of $m^1(t)$ and $m^2(t)$ with $\RR^d \ni x \mapsto \phi_{\delta}(x)=\sqrt{2\pi}\delta \exp{(-|x|^2/(2 \delta^2))}\in \RR$.

When $\nu>0$ and $\bar{m}_0^{\ell}$ ($\ell=1,2$) are sufficiently regular, the existence of classical solutions to \eqref{shelling} can be proved by standard methods (see \cite[Theorem 12]{AcBarCi17}, where the proof is provided when the space domain in \eqref{shelling} is bounded and Neumann boundary conditions are imposed on its boundary). 

In order to write \eqref{shelling} as $(FPK)$, note that by standard arguments in stochastic control theory (see e.g. \cite{MR2179357}) the first and second equations in \eqref{shelling} are equivalent to
\be\label{trajectorial_interpretation_hjbs}\ba{rcl}
v^{1}(x,t)&=& \inf_{\alpha_1} \; \EE\left( \int_{t}^{T} \left[ \half |\alpha_{1}(s)|^2+ V_{\eta,\delta}\left(m^1, m^2, X_1^{x,t,\alpha_1}(s),s\right) \right] \dd s \right),\\[6pt]
v^{2}(x,t)&=& \inf_{\alpha_2} \; \EE\left( \int_{t}^{T} \left[ \half |\alpha_{2}(s)|^2+ V_{\eta,\delta}\left(m^2, m^1, X_2^{x,t,\alpha_2}(s),s\right) \right] \dd s \right),
\ea
\ee
where the expectation $\EE$ is taken in a complete probability space $(\Omega,\F, \PP)$ on which  two independent $d$-dimensional Brownian motion $W^1$ and $W^2$ are defined, the $\RR^{d}$-valued processes $\alpha_1$ and $\alpha_2$ are adapted to the  natural filtration generated by $W^1$ and $W^2$, respectively,  and they satisfy $\EE\left(\int_{0}^{T} |\alpha^\ell(t)|^2 \dd t \right)<\infty$ ($\ell=1$, $2$). Finally, the processes $X_{\ell}^{x,t,\alpha_\ell}$ ($\ell=1$, $2$) are defined as the unique  solutions of
\be\label{controlled_equation_mfg}
\dd X_\ell(s)= \alpha_\ell(s) \dd s + \sqrt{2\nu} \dd W^\ell(s) \, \; s\in (t,T),  \hspace{0.8cm} X_\ell(t)=x.
\ee
By a verification argument (see e.g. \cite[Chapter III, Section 8]{MR2179357}),   the optimal dynamics for the problems defining $v^\ell$ ($\ell=1$, $2$) are given by the solutions of
$$
\dd X_\ell(s)=- \nabla v^\ell(X_\ell(s),s) \dd s + \sqrt{2\nu} \dd W^\ell(s) \, \; s\in (t,T),  \hspace{0.8cm} X_\ell(t)=x.
$$
Therefore, redefining $v^{\ell}: C([0,T]; \P_1(\RR^d))^2 \times \RR^d \times [0,T] \to \RR$ as
\be\label{trajectorial_interpretation_hjbs_depending_on_mu}\ba{rcl}
v^{1}(\mu^1,\mu^2,x,t)&=& \inf_{\alpha_1} \; \EE\left( \int_{t}^{T} \left[ \half |\alpha_{1}(s)|^2+ V_{\eta,\delta}\left(\mu^1, \mu^2, X_1^{x,t,\alpha_1}(s),s\right) \right] \dd s \right),\\[6pt]
v^{2}(\mu^1,\mu^2,x,t)&=& \inf_{\alpha_2} \; \EE\left( \int_{t}^{T} \left[ \half |\alpha_{2}(s)|^2+ V_{\eta,\delta}\left(\mu^2, \mu^1, X_2^{x,t,\alpha_2}(s),s\right) \right] \dd s \right),
\ea
\ee
we have that \eqref{shelling}, with $V_{\eta, \delta}$ instead of $V$ on the right hand side of the first and second equations, is equivalent to $(FPK)$ with $d_1=r_1=d_2=r_2=d$ and 
\be\label{coefficient_mfg_two_populations}
b^{\ell}(\mu^1,\mu^2,x,t)= - \nabla v^{\ell}(\mu^1,\mu^2,x,t) \; \ \mbox{and } \; \sigma^{\ell}(\mu^1,\mu^2,x,t)= \sqrt{2\nu} I_{d\times d}, 
\ee
where $I_{d\times d}$ is the $d\times d$ identity matrix. Arguing as in \cite{Carlini_Silva_2017} for the one population case, if $\nu>0$, it is easy to prove that for these drift terms, assumptions {\bf(H)} and {\bf(Lip)} are satisfied. 

As \eqref{trajectorial_interpretation_hjbs} shows, at the equilibrium \eqref{shelling} a typical player of population $\ell$ minimizes a cost that penalizes its speed, modelled by the quadratic penalization on $\alpha^\ell$, as well as a cost depending of its position, and the distribution of his and the other populations. Recalling that $V_{\eta,\delta}$ is an approximation of $V$, defined in \eqref{v_non_regularized}, the cost $V_{\eta,\delta}$ models a xenophobia effect (the regularization of the first term in $V$) and penalizes   overcrowded regions  taking into account the sum of both populations (the regularization of the second term in $V$). 

Note that the coefficients $b^\ell$ in \eqref{coefficient_mfg_two_populations} depend on the value functions $v^\ell$, which do not admit an explicit expression. Moreover, as \eqref{trajectorial_interpretation_hjbs_depending_on_mu} shows,  $b^{\ell}(\mu^1,\mu^2,x,t)$ depends on the values $(\mu^1(s), \mu^2(s))$ with $s \in (t,T)$, and so the scheme \eqref{scheme_nonlinear_stochastic_case_multiple_populations} is implicit (see Remark \ref{numerical_inefficiency_one_fp}{\rm(ii)}). In order to obtain an implementable scheme, we approximate $b$ by computable vector fields. More precisely,  we use a Semi-Lagrangian scheme to approximate $v^1$ and $ v^2$, as described in \cite{CS15} and in Section 5.3 of \cite{Carlini_Silva_2017}  for the case of  a single population. 
We then call $v^{\ell,\rho,h}: C([0,T]; \P_{1}(\RR^d))^2  \times \RR^d \times [0,T] \to \RR$ ($\ell=1,2$) the resulting interpolated discrete value functions and
we  regularize them by using    space convolution
$$
v^{\ell, \rho,h,\eps}[\mu^1,\mu^2](\cdot, t):= \phi_{\eps} \ast v^{\ell, \rho,h}[\mu](\cdot,t) \hspace{0.3cm} \forall \; t\in [0,T],
$$
 where $\phi_{\eps}(x)=\sqrt{2\eps}\delta \exp{(-|x|^2/(2 \eps^2))}$.
Next, we approximate the drifts in \eqref{coefficient_mfg_two_populations} by
 $$b^{\ell}[\mu^1,\mu^2](x,t):=- \nabla_{x} v^{\ell,\rho,h,\eps}[\mu^1,\mu^2](x,t).$$
Consider sequences $\rho_n$, $h_n$ and $\eps_n$    converging to $0$ as $n\to \infty$ and define $(m_n^1,m_n^2)$ as the sequence in $C([0,T];\P_1(\RR^d))^2$ constructed with the scheme \eqref{scheme_nonlinear_stochastic_case_multiple_populations} and the extension \eqref{linear_interpolation_extension_of_simplex}, by considering discrete characteristics computed with the drifts $b^\ell_{n}[\mu^1,\mu^2](x,t):=- \nabla_{x} v^{\ell,\rho_n,h_n,\eps_n}[\mu^1,\mu^2](x,t)$. Then, arguing   exactly as in \cite[Section 5.3]{Carlini_Silva_2017}, we can prove that if $\nu>0$ and  $\rho_n^2= o(h_n)$ and $\rho_n=o(\eps_n)$ and $(\mu_n^1, \mu_n^2) \to (\mu^1, \mu^2)$ in  $C([0,T];\P_1(\RR^d))^2$ we have that {\bf(H')} is satisfied. Therefore, we can apply Theorem \ref{convergence_with_coefficient_approximation} to deduce that $(m_n^1,m_n^2)$  admits at least one limit point $(m^1,m^2)$ and every such limit point solves $(FPK)$.  When $\nu=0$, the situation is more delicate because  we need to construct approximations which are absolutely continuous  with respect to the Lebesgue measure (see  \cite[Remark 4.2 and Remark 5.1{\rm(ii)}]{Carlini_Silva_2017}). The resulting scheme is the natural extension of the one proposed in \cite{MR3148086} to the multipopulation case. Arguing as in the proof of Theorem 3.12 in \cite{MR3148086}, we can obtain a convergence result under the additional assumptions that $d=1$ and $h_n=o(\eps_n)$.

\subsubsection{ Numerical tests } \label{testshelling}
As in \cite[Section 6.2.1]{AcBarCi17}, we solve system \eqref{shelling}, with $V$ replaced by $V_{\eta,\delta}$, on the one dimensional space domain $\Omega=[-0.5,0.5]$. We set the final time $T=4$ and we consider homogeneous Neumann boundary conditions.
The initial densities are given by 
$$ m^1(x,0) = 3/4+1/2 \mathbb{I}_{[-1/2,-1/4]\cup [0,1/4]}(x) \quad {\textrm {and}}\quad  m^2(x,0) = 3/4+1/2 \mathbb{I}_{ [-1/4,0]\cup[1/4,1/2]}(x),$$
where for $A\subseteq \RR$, $\II_A(x)=1$ if $x\in A$ and $\II_A(x)=0$, otherwise.  We choose $\rho=0.02$ and $h=\rho^\frac{3}{2}$. 
The regularizing parameters are set to $\delta=\eps=0.025$ and  $\eta=10^{-5}$. 

In order to compute the solution of the fully discrete system, we have used  the learning procedure  proposed in \cite{CH17} in the continuous framework. We point out that a rigorous study of the convergence of this method for the resolution of discretizations of MFG systems has not been established yet and remains as an interesting research subject. We stop the procedure when the difference between two successive discrete densities,  measured in the maximum discrete norm,  is smaller than  $5\times 10^{-3}$. 

Due to the symmetry of the initial conditions and to the form of the coupling terms, the evolutions of the two populations are symmetric to each other. This symmetry can be observed in all the simulations. We also observe that the evolutions present a {\it turnpike} property since most of the time after and before the  $t=0$ and $t=T=4$, respectively,  the distribution is near a stationary configuration.

In Figure \ref{Test2a}, computed with $\nu=0.05$, we show the evolution  of the two densities  at  the  times $t=0$, $0.1$, $0.5$, $2$,  $3$, $4$. We can observe that the two densities separate from each other, with only a small overlap region at the end. We also observe that the  configurations at times $t=2$ and $t=3$ have the same shape,  which is near a stationary configuration (see \cite[Section 6.1]{AcBarCi17}). For this viscosity parameter, our results are almost identical to those in  \cite[Section 6.2.1, Figure 8]{AcBarCi17}.

In Figure \ref{Test2b}, computed with $\nu=0.001$,  we show the configuration at  the  times $t=0$, $0.1$, $0.2$, $1$,  $2$, $4$. The two densities  separate faster than the previous case, reaching a nearly steady-state solution already at time $t=1$. We can observe that the resulting segregated configurations    differ considerably from the the previous case, computed with $\nu=0.05$. 

In Figure \ref{Test2c}, computed with $\nu=0$, we show the configuration  of the two measures at  the  times $t=0$, $0.1$, $0.5$, $1$, $2$, $4$. As expected in the deterministic case, the evolution is much less smooth. Compared to the diffusive cases, at the final time $T$, the supports of the densities $\mathbf{m}_{\rho,h}^1$ and $\mathbf{m}_{\rho,h}^2$  are disjoint and separated by  much larger sets. We insist that, for the previous and the current tests, the solutions captured by the scheme  differ importantly from the ones computed with larger viscosity parameters (see Figure \ref{Test2a} and \cite[Section 6.2.1]{AcBarCi17}). 

\begin{figure}[ht!]
\begin{center}
\includegraphics[width=4cm]{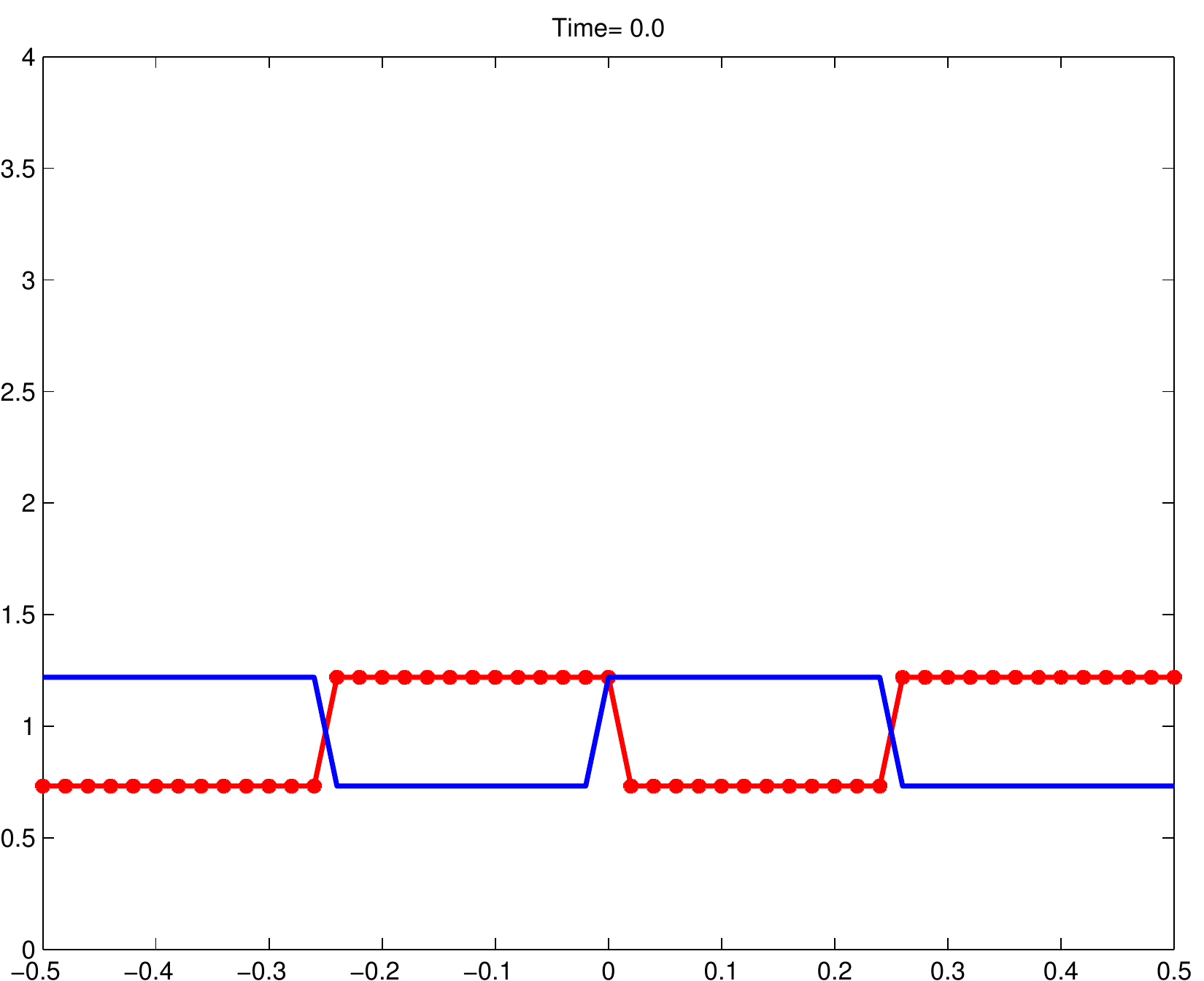}\includegraphics[width=4cm]{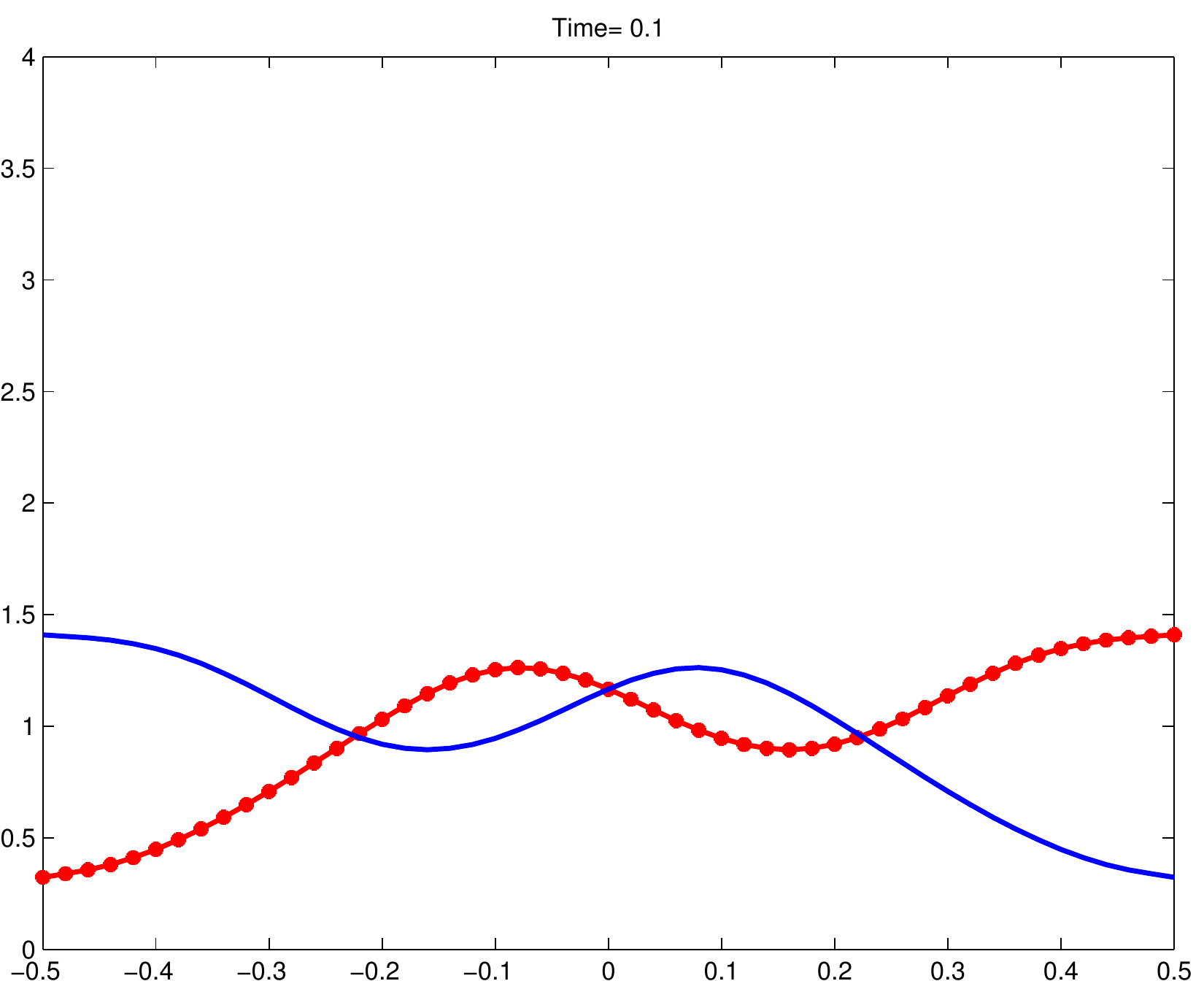}\includegraphics[width=4cm]{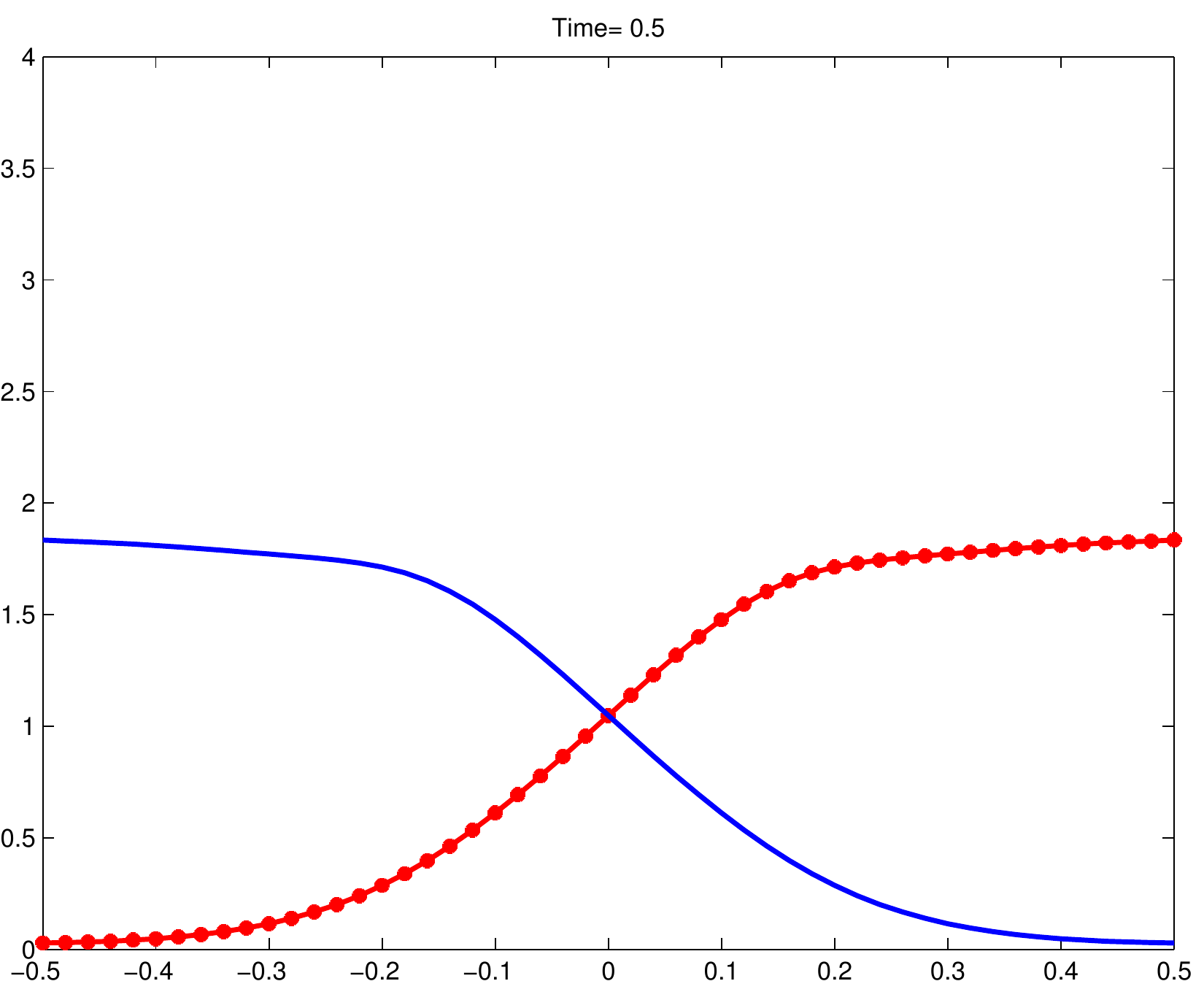}\\
\includegraphics[width=4cm]{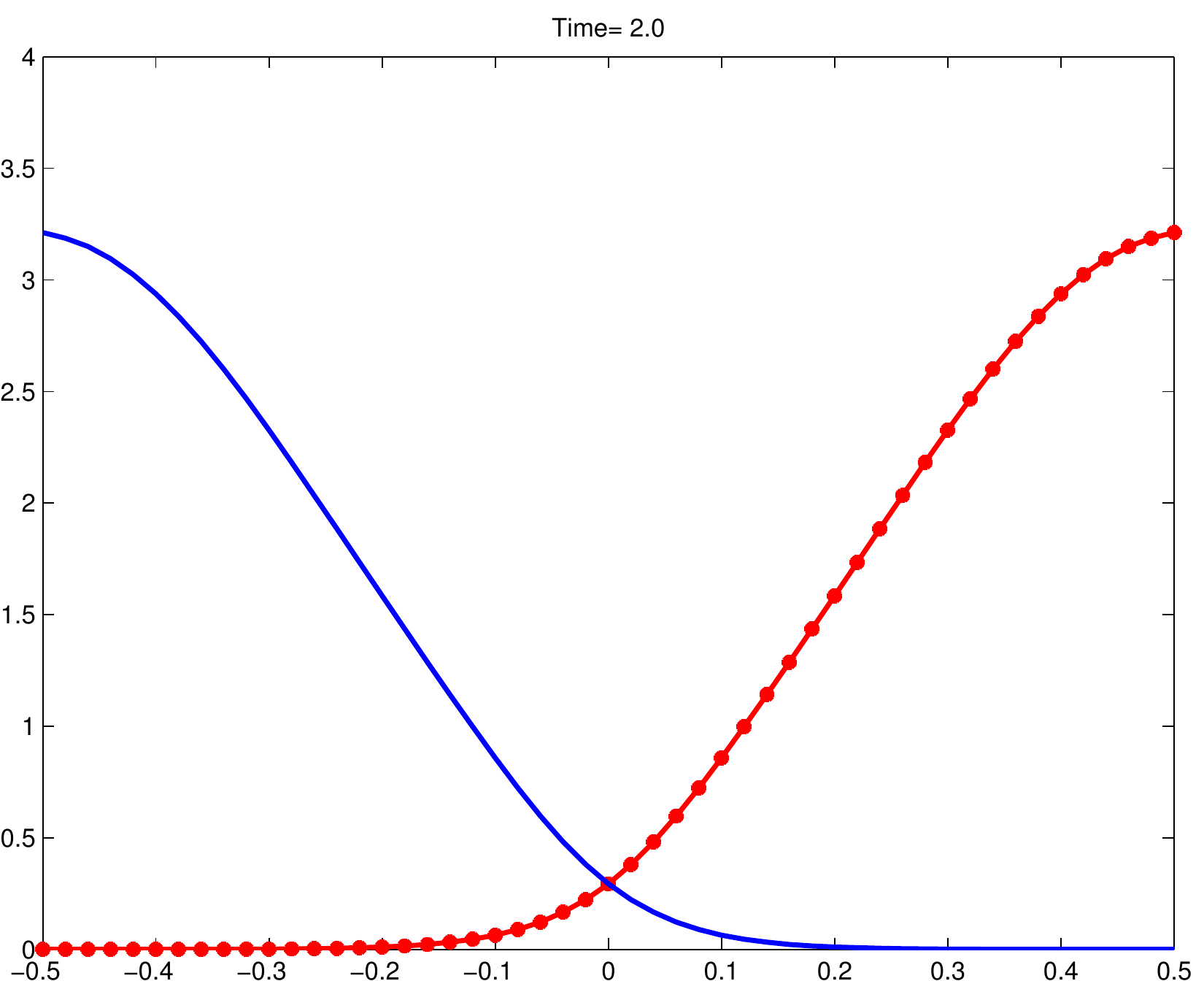}\includegraphics[width=4cm]{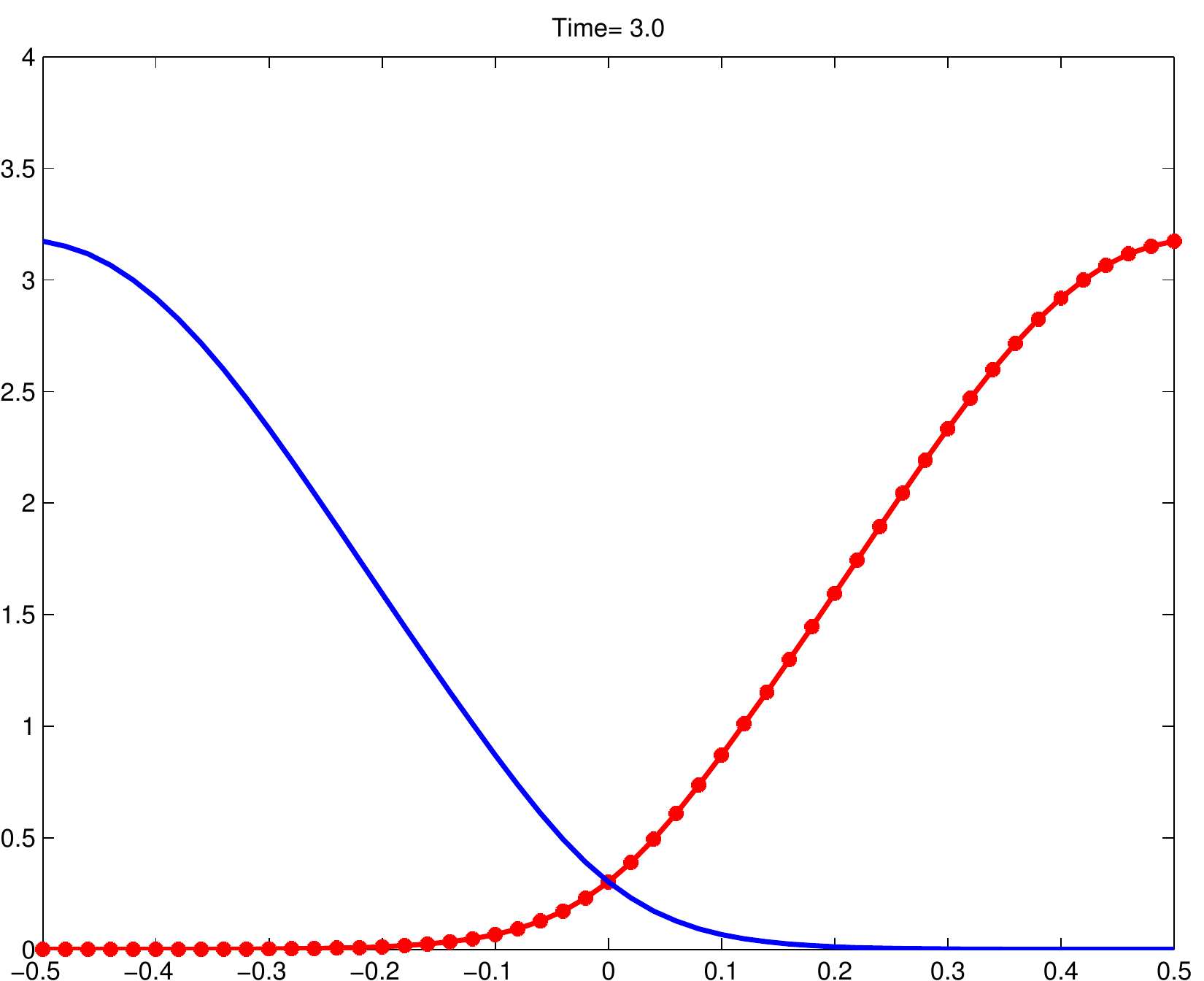}\includegraphics[width=4cm]{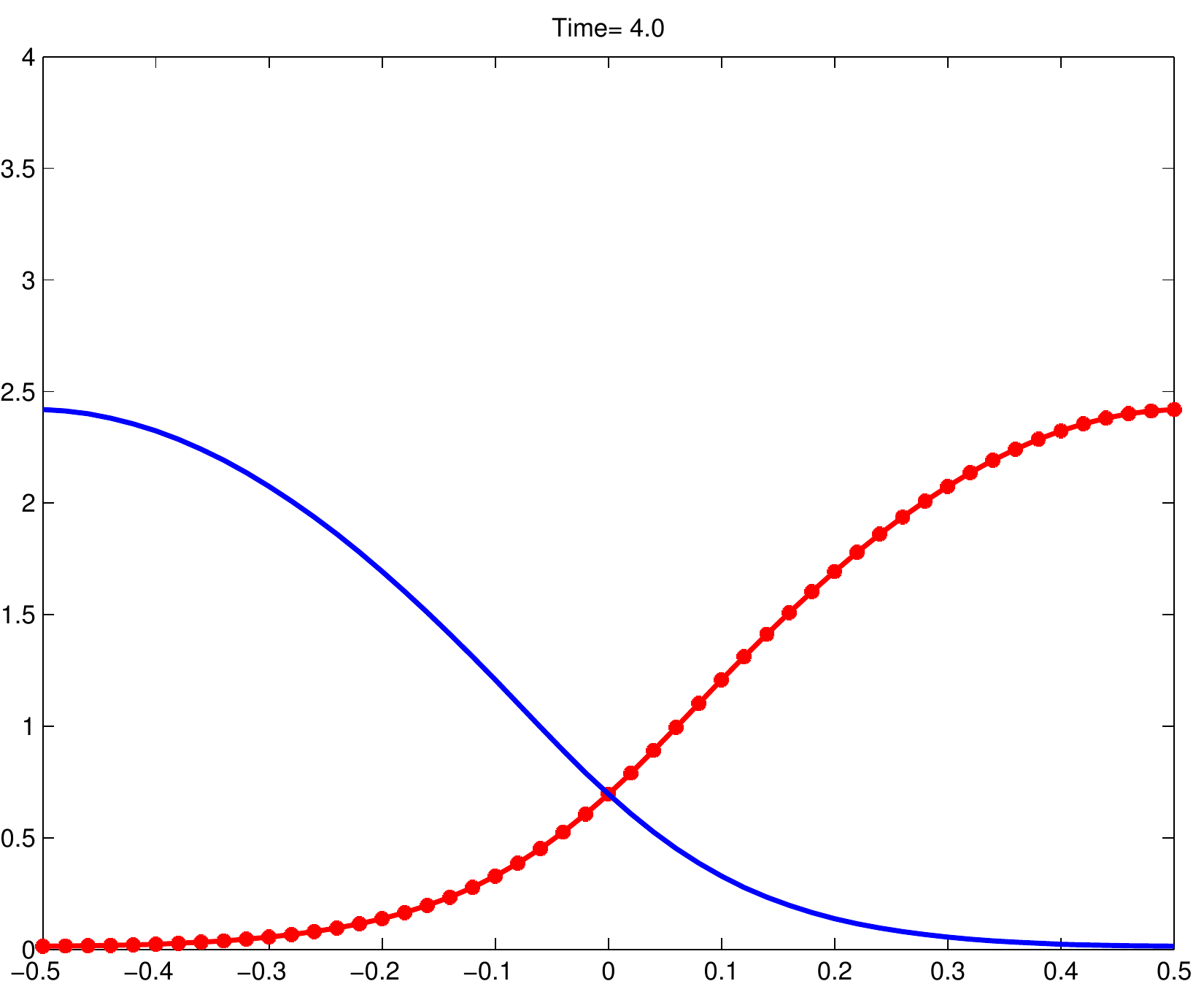}
{\caption{ Evolution of $\mathbf{m}_{\rho,h}^1$({\color{blue}-}) and $\mathbf{m}_{\rho,h}^2$ ({\color{red}{$-*$}}) at the time  $t=0$, $0.1$, $0.5$, $2$, $3$, $4$ computed with $\nu=0.05$.}\label{Test2a}}
\end{center}
\end{figure}
  \begin{figure}[ht!]
\begin{center}
\includegraphics[width=4cm]{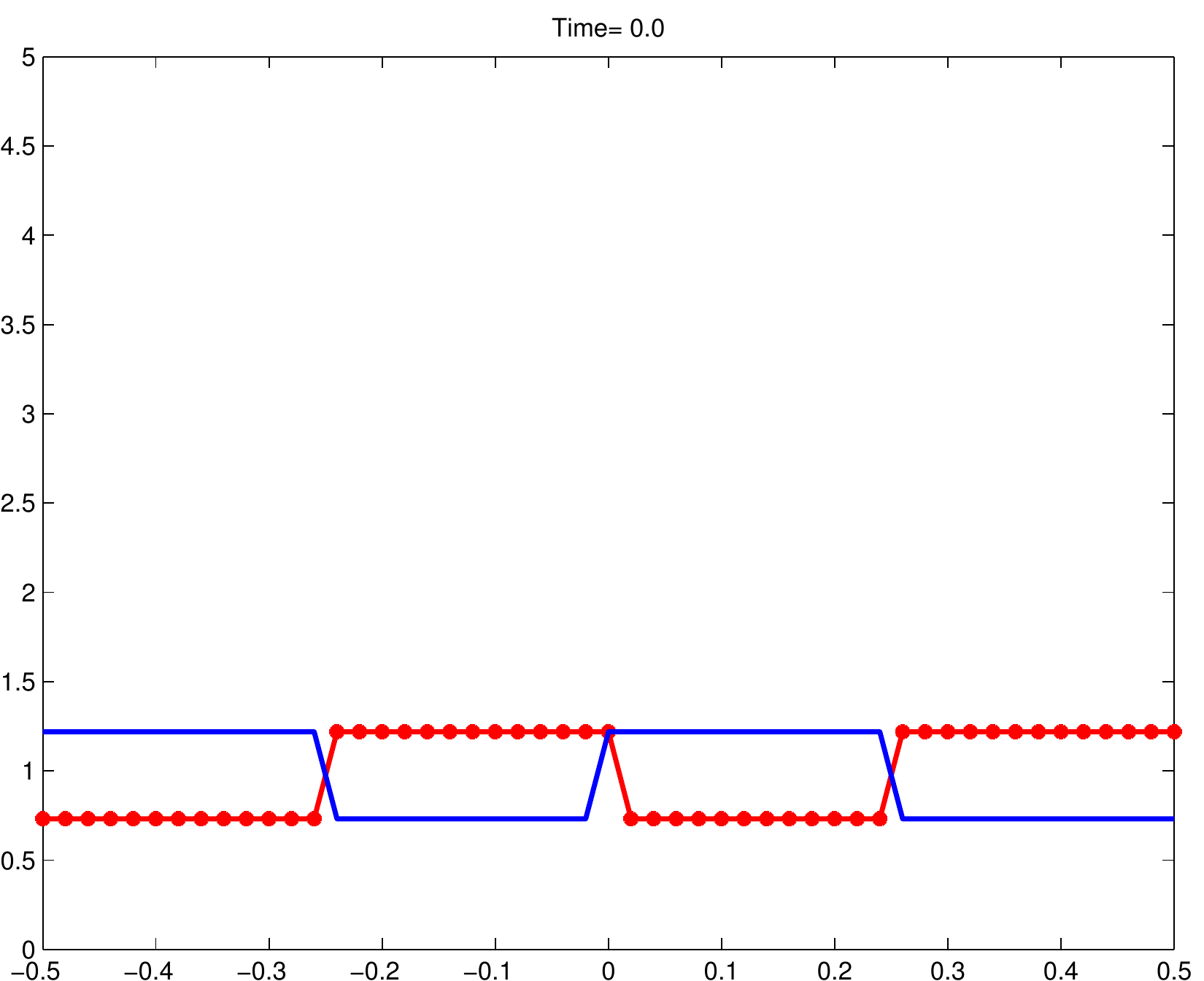}\includegraphics[width=4cm]{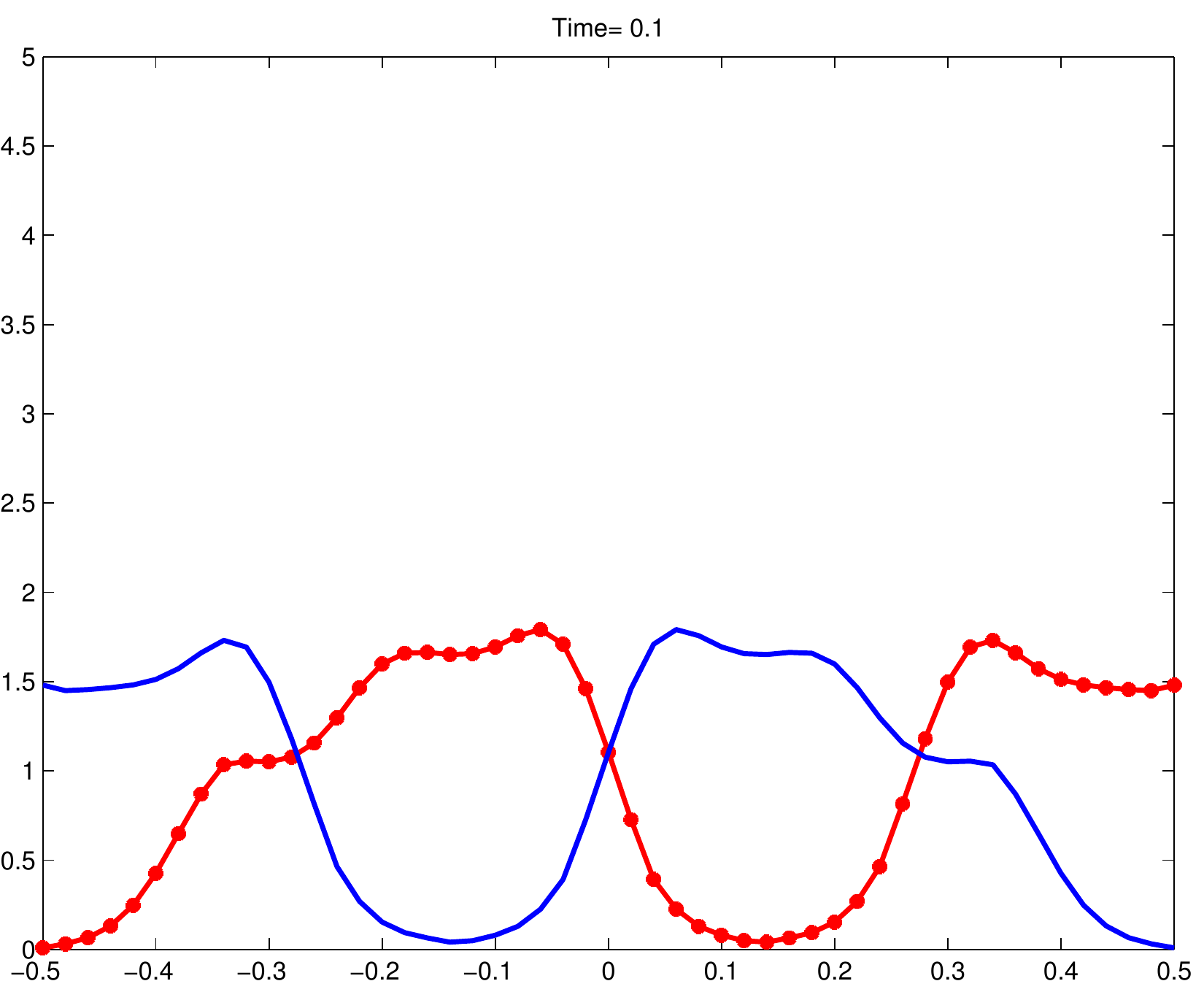}\includegraphics[width=4cm]{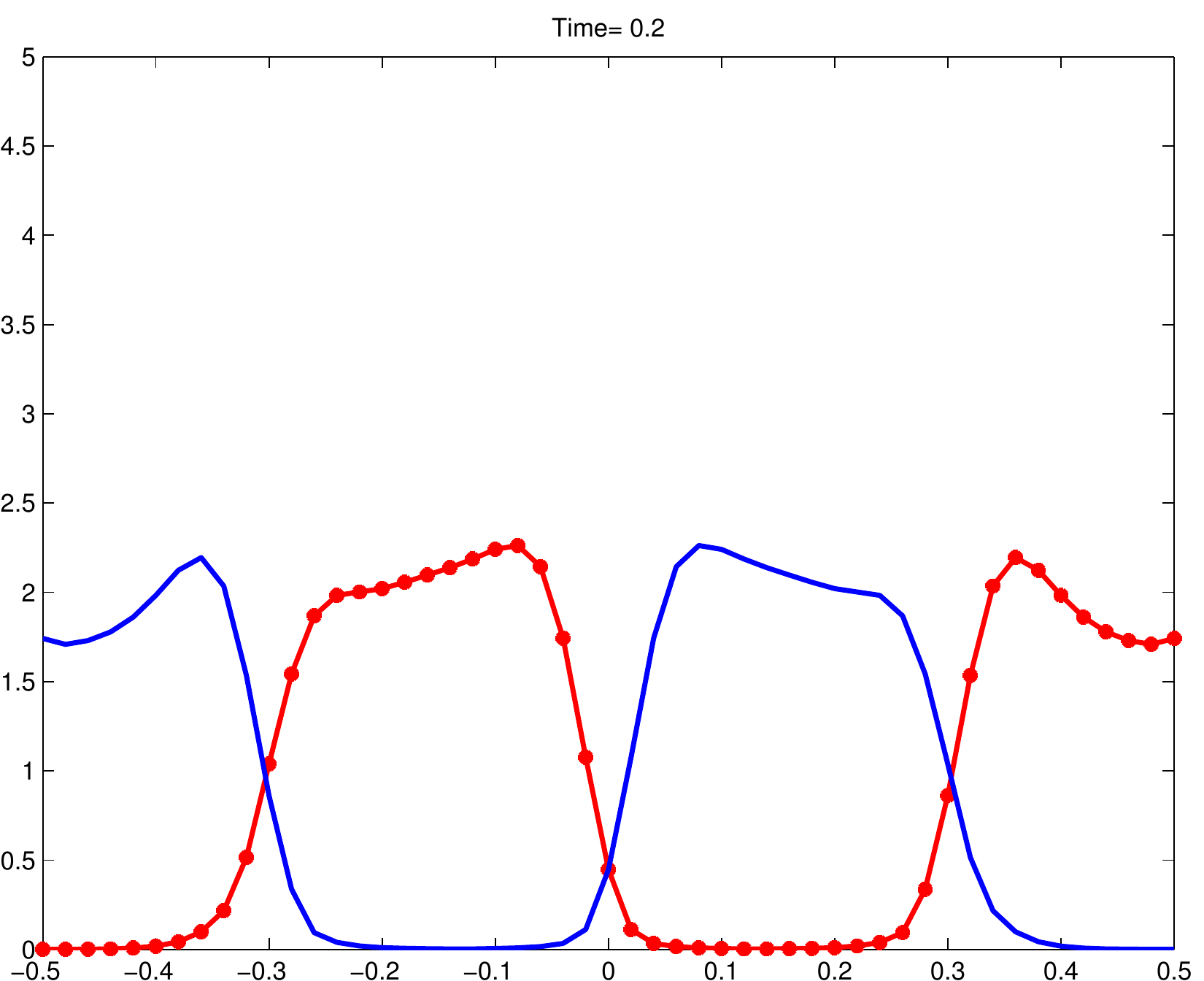}\\
\includegraphics[width=4cm]{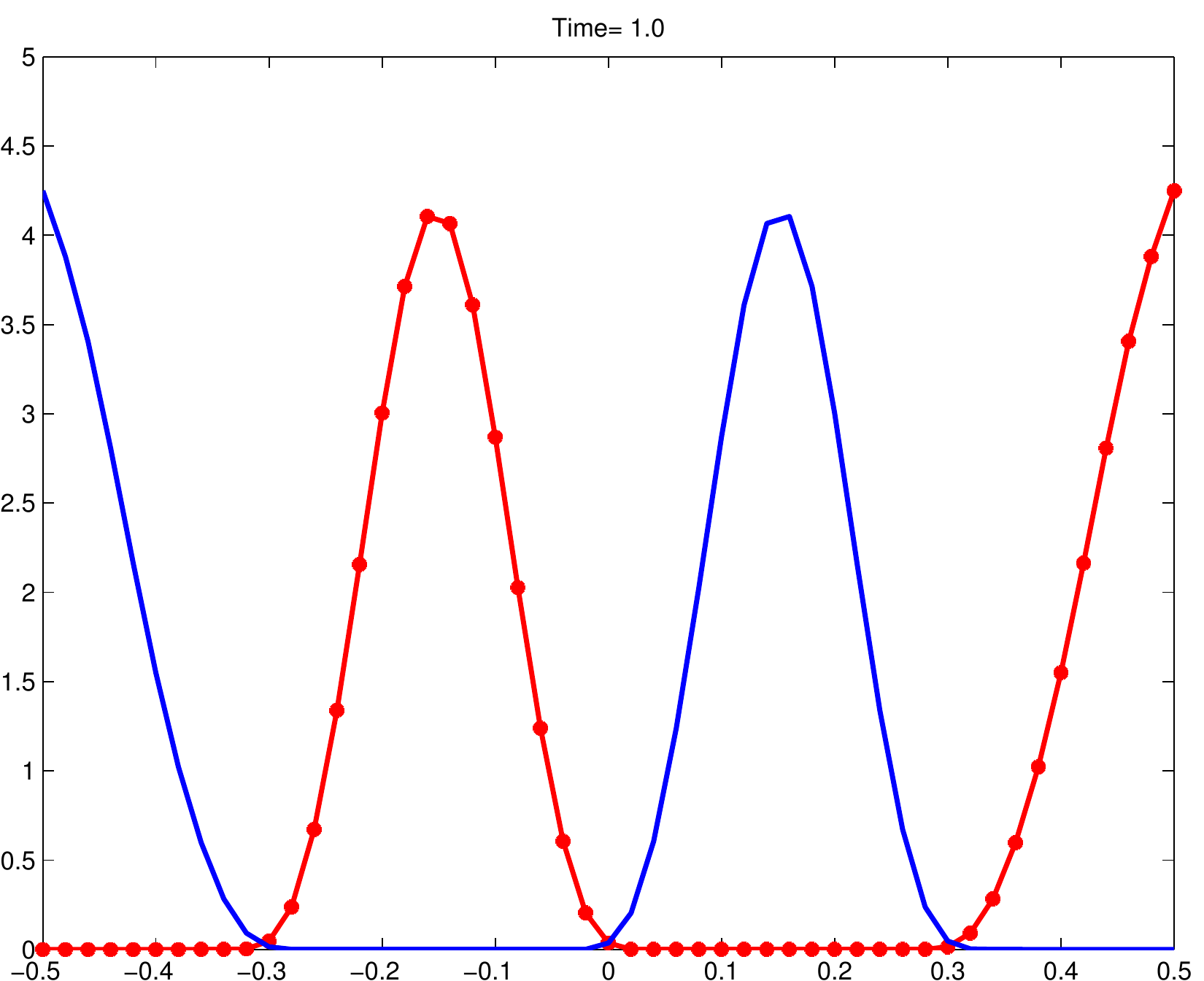}\includegraphics[width=4cm]{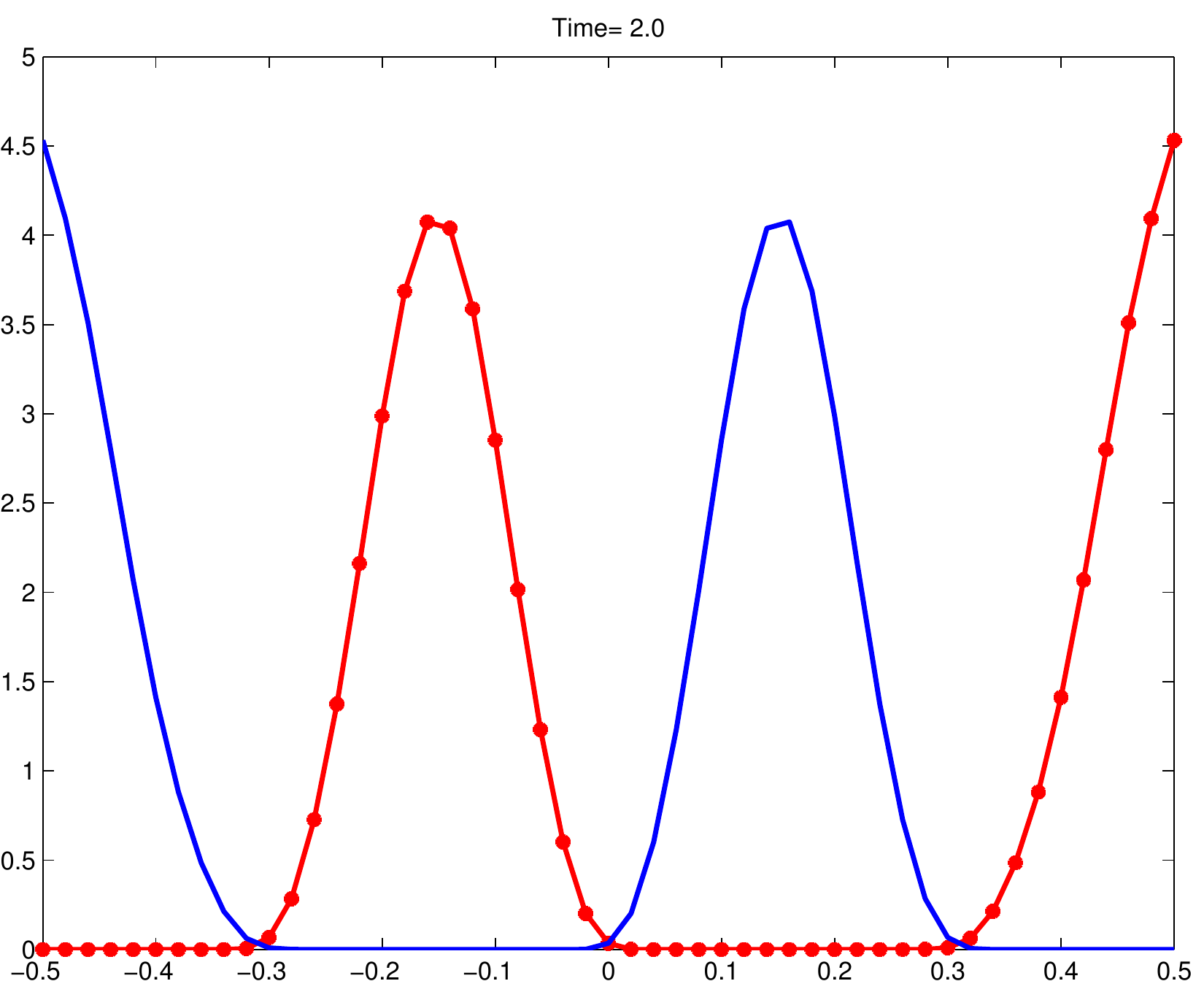}\includegraphics[width=4cm]{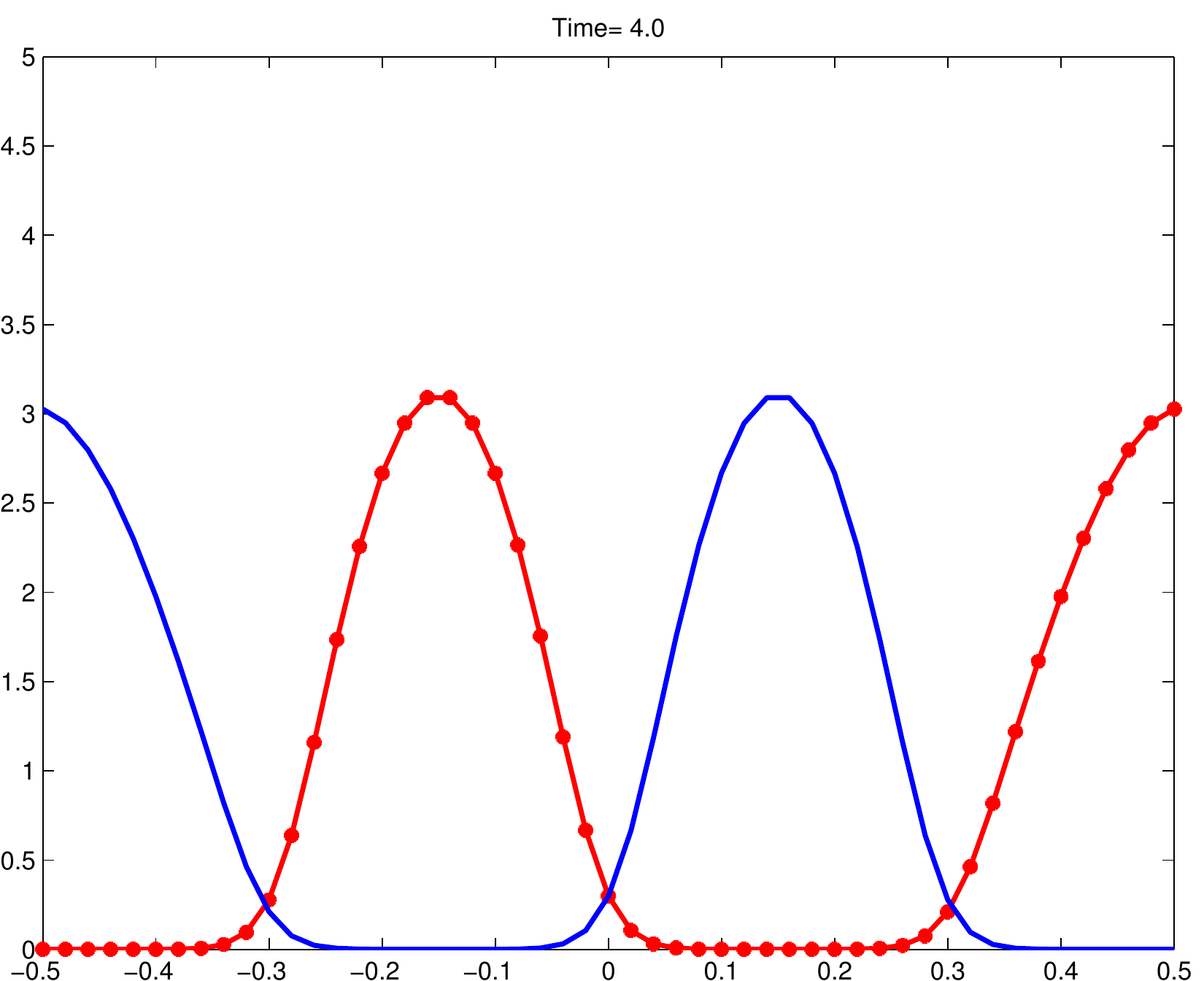}
{\caption{ Evolution of $\mathbf{m}_{\rho,h}^1$({\color{blue}-}) and $\mathbf{m}_{\rho,h}^2$ ({\color{red}{$-*$}}) at the times  $t=0 , 0.1, 0.2, 1, 2, 4$ with  $\nu=0.001$.}
\label{Test2b}}
\end{center}
\end{figure}
\begin{figure}[ht!]
\begin{center}
\includegraphics[width=4cm]{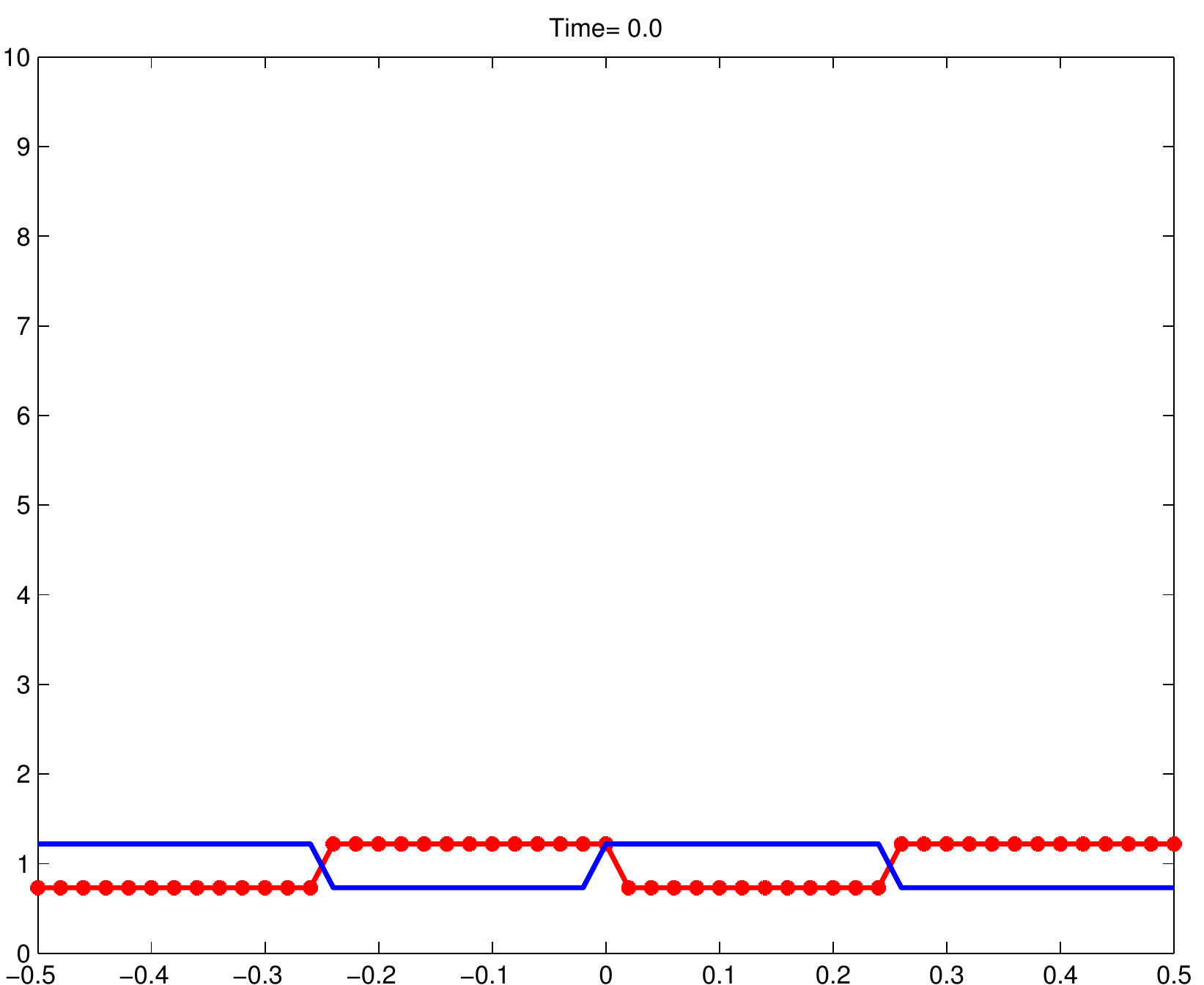}\includegraphics[width=4cm]{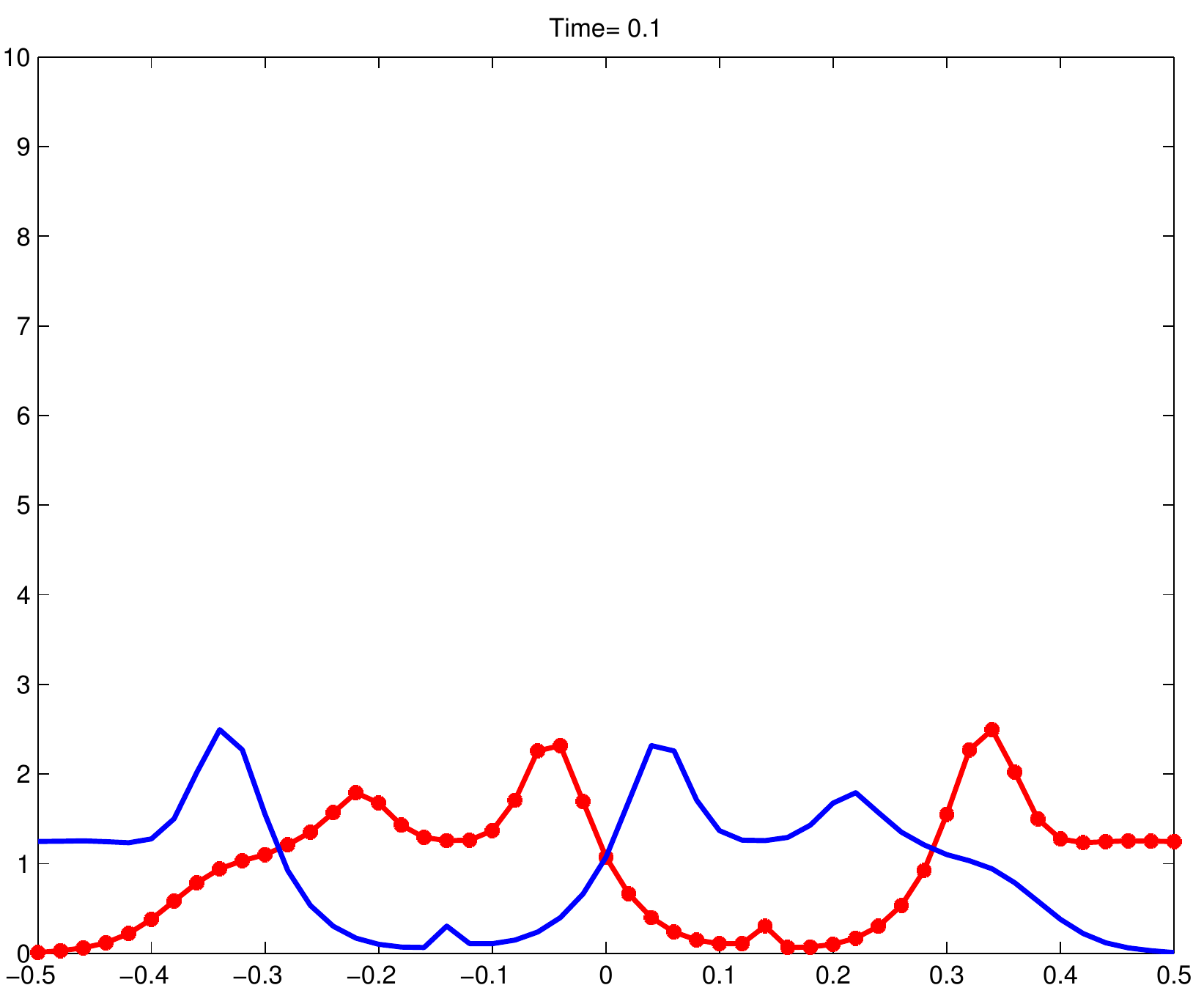}\includegraphics[width=4cm]{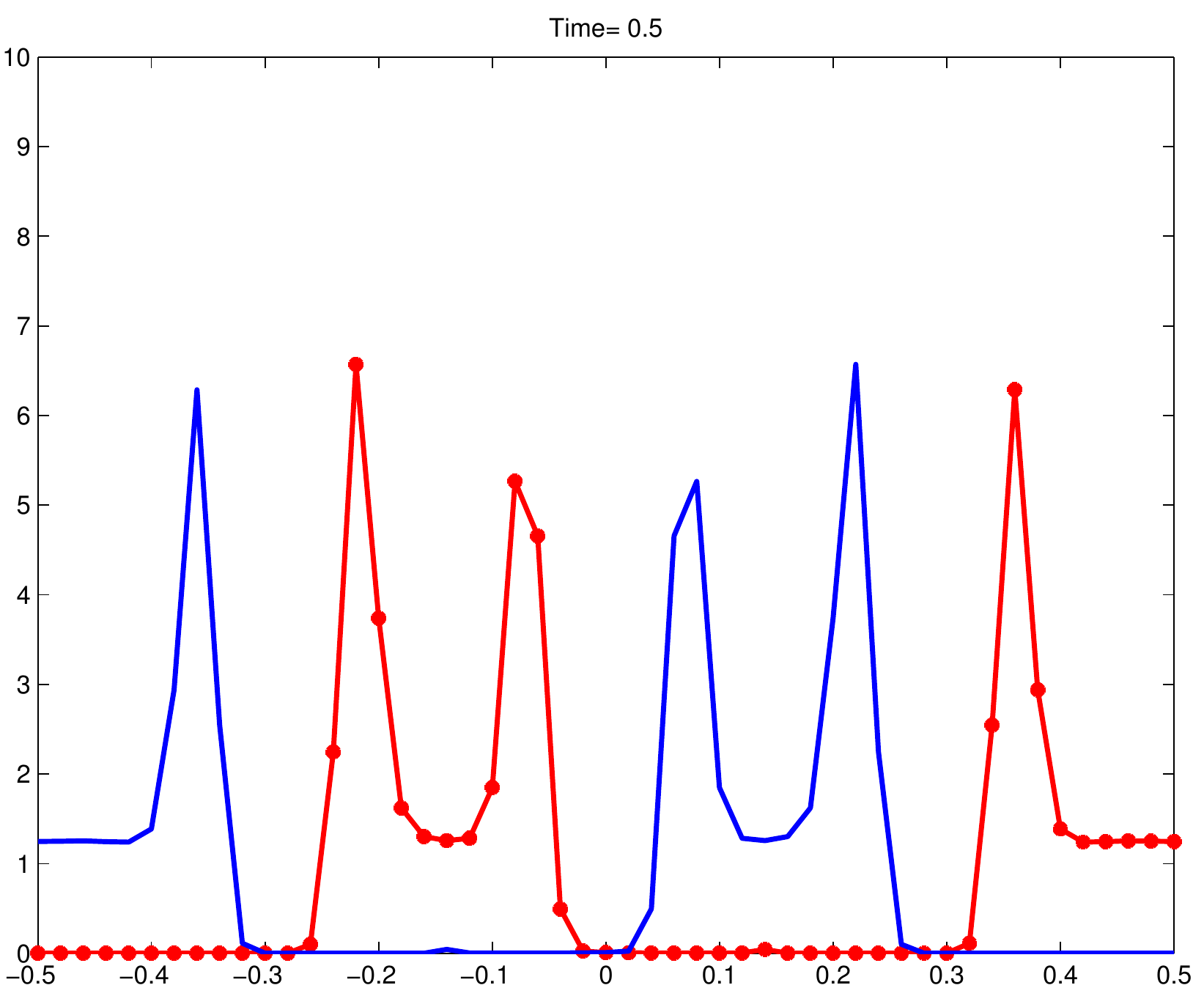}\\
\includegraphics[width=4cm]{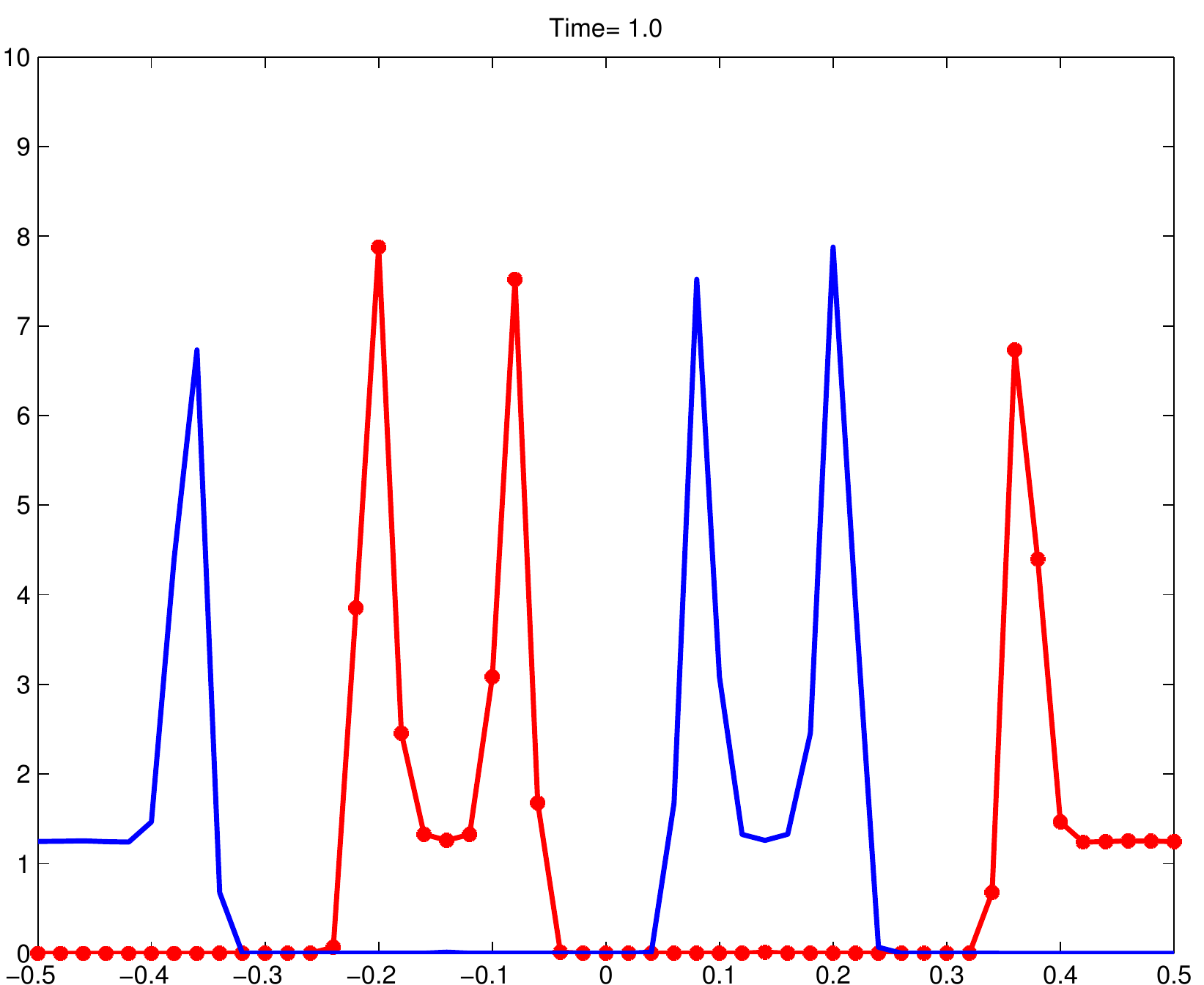}\includegraphics[width=4cm]{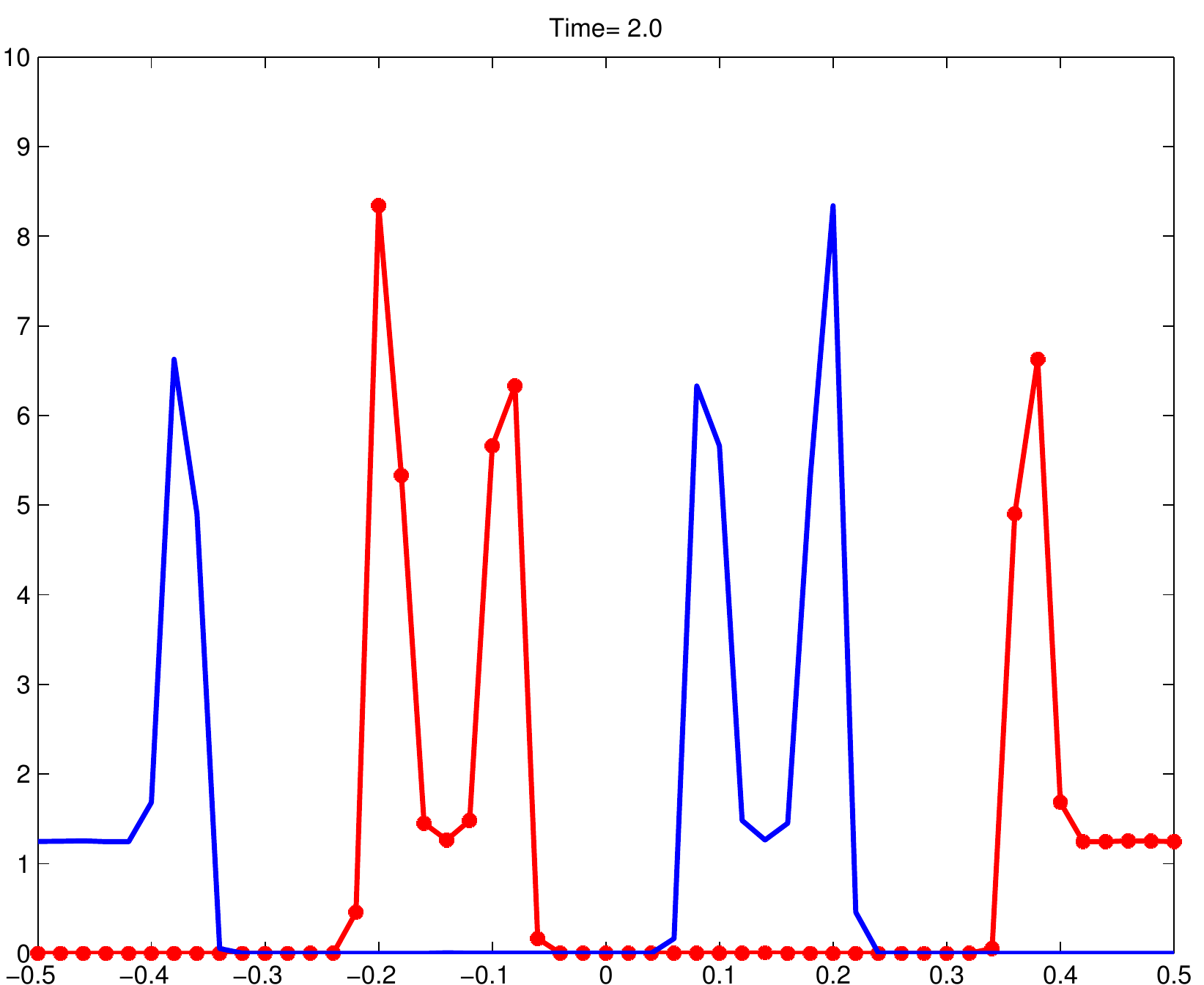}\includegraphics[width=4cm]{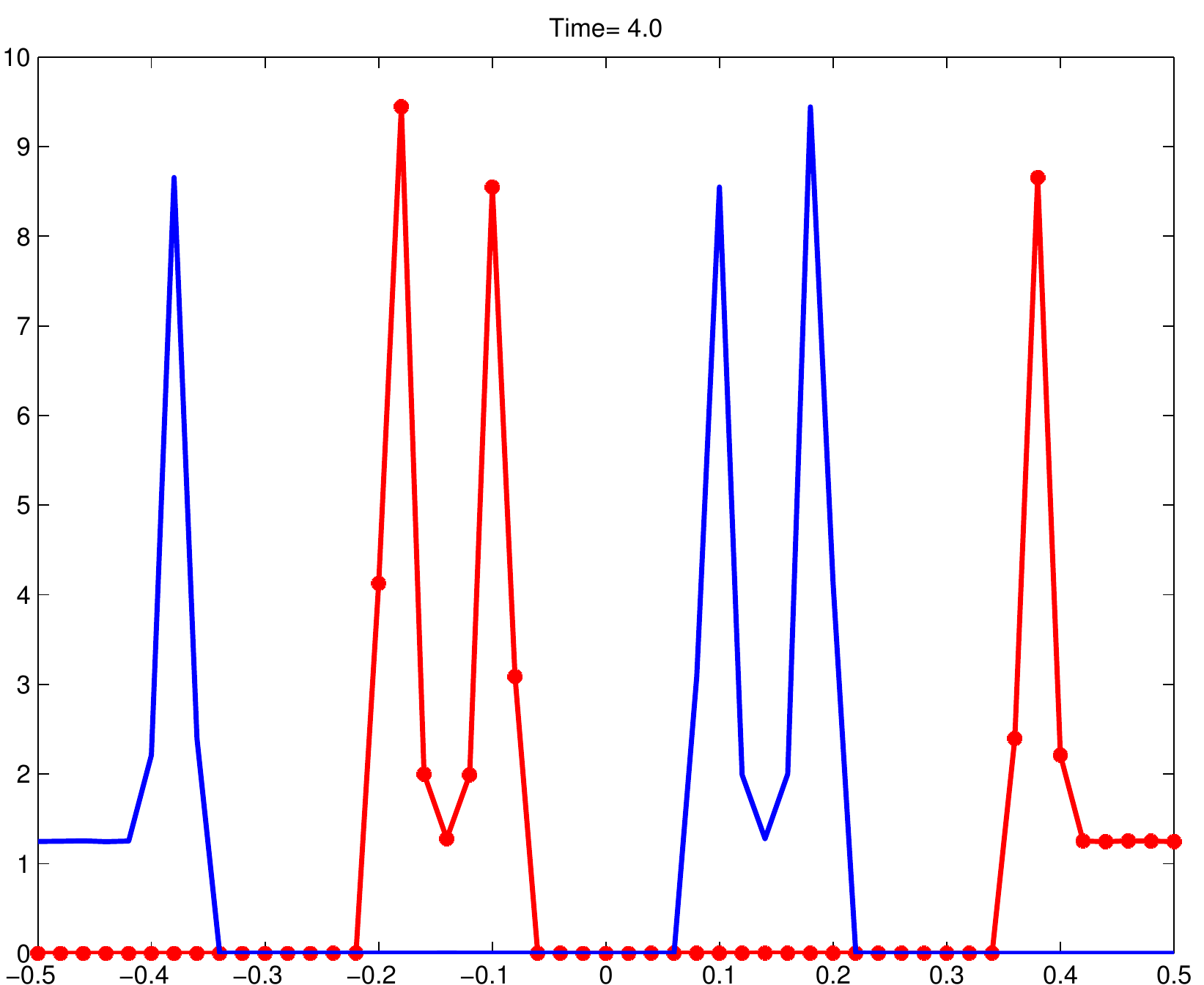}
{\caption{ Evolution of $m^1_{\rho,h}$({\color{blue}-}) and $m^2_{\rho,h}$ ({\color{red}{$-*$}})  at the time  $t=0$, $0.1$, $0.5$, $1$, $2$, $4$ computed with  $\nu=0$.}
\label{Test2c}}
\end{center}
\end{figure}

\bibliographystyle{plain}
\bibliography{bibFP}
\end{document}

%% file: mymacro.tex
\def\dd{{\rm d}}

\def\weight(#1,#2){c_{#1,#2}}

 \if {
  
 } \fi

\if {

}{{\caption}}
} \fi

\def\B{\mathcal{B}}

\def\F{\mathcal{F}}
\def\G{\mathcal{G}}

\def\K{\mathcal{K}}

\def\M{\mathcal{M}}

\def\P{\mathcal{P}}

\def\T{\mathcal{T}}

\def\eps{\varepsilon}

\def\half{\mbox{$\frac{1}{2}$}}

\def\1B{{\bf  1}}

\newcommand{\NN}{\mathbb{N}}
\newcommand{\ZZ}{\mathbb{Z}}

\newcommand{\RR}{\mathbb{R}}

\def\II{\mathbb{I}}
\def\EE{\mathbb{E}}
\def\PP{\mathbb{P}}

\newcommand\be{\begin{equation}}
\newcommand\ee{\end{equation}}
\newcommand\ba{\begin{array}}
\newcommand\ea{\end{array}}
\newcommand{\bean}{\begin{eqnarray*}}
\newcommand{\eean}{\end{eqnarray*}}

